\documentclass[12pt]{article}
\usepackage{amsmath}
\usepackage{latexsym}
\usepackage{amssymb}
\usepackage{xcolor}
\usepackage{url}
\usepackage{mathrsfs}
\usepackage{pdflscape}
\usepackage{tikz}
\usepackage{rotating}

\newtheorem{theorem}{Theorem}[section]
\newtheorem{lemma}[theorem]{Lemma}
\newtheorem{definition}[theorem]{Definition}
\newtheorem{Remark}[theorem]{Remark}
\newtheorem{Conjecture}[theorem]{Conjecture}
\newtheorem{proposition}[theorem]{Proposition}
\newtheorem{corollary}[theorem]{Corollary}
\newtheorem{Example}[theorem]{Example}
\newtheorem{question}[theorem]{Question}
\newtheorem{Number}[theorem]{\!\!}

\newenvironment{example}{\begin{Example}\rm}{\end{Example}}
\newenvironment{remark}{\begin{Remark}\rm}{\end{Remark}}

\newenvironment{numba}{\begin{Number}\rm}{\end{Number}}
\newenvironment{proof}{{\noindent\bf Proof.}}%
                  {\nopagebreak\hspace*{\fill}$\Box$\medskip\par}

\DeclareMathOperator{\Aut}{Aut}
\DeclareMathOperator{\Isom}{Isom}

\DeclareMathOperator{\Sym}{Sym}
\DeclareMathOperator{\Alt}{Alt}
\newcommand{\tree}{T}
\DeclareMathOperator{\id}{id}

\DeclareMathOperator{\con}{{\sf con}}

\newcommand{\vtx}[1]{\mathbf{#1}}
\newcommand{\tvtx}[1]{\tilde{\mathbf{#1}}}
\newcommand{\gr}[1]{\mathcal{#1}}
\newcommand{\Psr}[1]{\widehat{#1}}

\def\tree{T}
\def\treeq{T_{q+1}}
\def\wtreeq{{T}_{q+1}}
\def\treeqq{T_{q,q}}
\def\treeqqo{T_{q,q-1}}

\def\inftran{\mbox{\rm Sc}(q+1)}

\title{Scale groups}
\author{George A. Willis\thanks{This research was supported by the Australian Research Council grant FL170100032 and by a Texas A\&M, College Station, visitor grant.}}
\begin{document}
\maketitle
\begin{abstract}
Closed subgroups of the group of isometries of the regular tree $\treeq$ that fix an end of the tree and are vertex-transitive are shown to correspond, on one hand, to self-replicating groups acting on rooted trees and, on the other hand, to elements of totally disconnected, locally compact groups having positive scale.
\end{abstract}
{\bf Classification:} Primary 20E08; 
Secondary 22D05, 20E22, 20F65\\
%
%
\\[11mm]
\section{Introduction}
\label{sec:Intro}


This article concerns what we shall call \emph{scale groups}, that is, closed subgroups of the group of isometries of a regular tree fixing an end of the tree and transitive on its vertices. Scale groups are totally disconnected, locally compact (t.d.l.c.)~groups. They appear in several guises in group theory, and emerge as an essential part of the abstract theory of t.d.l.c.~groups. The article discusses the roles played by scale groups and develops the basic properties of these groups. The overview presented in the next few paragraphs refers to the table in Figure~\ref{fig:concept_relations}. Rows of the table correspond to increasing levels of concrete representation of groups, and columns to three settings in which scale groups occur.

	If $G$ is any closed group of isometries of a tree acting $2$-transitively on the boundary, then, by~\cite[Proposition~2]{Nebbia} or~\cite[Lemma~3.1.1]{BM} and~\cite[Lemma~2.2]{Radu}, $G$ has at most two orbits of vertices. In the case of one orbit, the fixator of an end is a scale group.
Examples include subgroups of simple Lie groups over local fields~\cite{Serre_Lie,Schneider} and of Kac-Moody groups over finite fields~\cite{CarErsRit,CapRem}. The affine buildings, see~\cite{BrownAbram}, of rank~1 groups in these classes are trees and the action of the group on the boundary is $2$-transitive. These scale groups are part of the ``microcosm" for the class of all (simple) t.d.l.c.~groups discussed by P.-E.~Caprace in~\cite{Caprace_microcosm}. Further examples of groups $2$-transitive on the boundary of a tree are defined in~\cite{BM} with the aid of a \emph{legal colouring} of the edges (an idea introduced in~\cite{LMRoom}), and this definition inspired the method for constructing new examples from old given in \cite{P(k)_BanksElderWi}. The Caprace microcosm is explored in~\cite{Radu}, which classifies, for certain valencies, simple groups of tree automorphisms acting $2$-transitively on the boundary. These classes of scale groups are identified in the first column of the table in Figure~\ref{fig:concept_relations}.

A scale group arises from any automorphism, $\alpha$, of a t.d.l.c.~group, $G$, for which no compact open subgroup, $U$, is invariant. For any  such $U$, the index $[\alpha(U) : \alpha(U)\cap U]$ is a positive integer and the minimum of these indices is the \emph{scale of $\alpha$}, denoted $s(\alpha)$. If no $U$ is invariant, then $s(\alpha)>1$ and there is a closed $\alpha$-stable subgroup, $H$, of $G$ such that $H\rtimes \langle\alpha\rangle$ is represented as a scale group acting on a tree with valency $s(\alpha)+1$,~\cite{GWil,BW}. Conversely, every scale group has the form $H\rtimes \langle\alpha\rangle$ with $H$ a closed subgroup and $\alpha$ an inner automorphism. Scale groups are thus essential features of the structure of t.d.l.c.~groups, analogous to $(ax+b)$-subgroups of a Lie group. These ideas are summarised in the second column of the table in Figure~\ref{fig:concept_relations}. The \emph{compatible labelling} referred to in the third row of the second column is important in the analysis of scale groups made here. 

Also mentioned in the second column, scale groups occur in the setting of actions t.d.l.c.~groups on hyperbolic spaces. These actions are classified in \cite{Gromov_hyperbolic} and, among the classes, the \emph{focal} actions turn out to be precisely scale groups, see~\cite[\S3.1 \& Theorem 7.1]{CaCoMoTe}, thus giving an alternative characterisation of general scale groups in a quite different context. 

Scale groups occur less directly in connection with just-infinite groups~\cite{Wilson_ji}, and with periodicity and growth in finitely generated discrete groups~\cite{Grigorchuk}. The notion of a \emph{branch group}, which acts on a locally finite rooted tree \cite{Grig_branch}, emerged from these distinct directions: in \cite{Wilson_ji} branch groups form one of two classes of just-infinite groups; and the examples constructed in \cite{Grigorchuk} of finitely generated, locally finite torsion groups with intermediate growth are branch groups. Since they act on locally finite rooted trees, branch groups are residually finite and embed into profinite groups, {\it i.e.\/}~compact t.d.l.c.~groups. Many branch groups are \emph{self-similar} or even \emph{self-replicating}, see~\cite[Definition~2.8.1]{Nekrash}, which properties are defined in terms of a labelling of vertices of the rooted tree by strings. Propositions~\ref{prop:s.s.&s.r.} and~\ref{prop:s.s.&s.r.2} in the present paper show that there is an equivalence between and self-replicating groups and scale groups via a compatible labelling of the regular tree. These connections are summarised in the third column of the table in Figure~\ref{fig:concept_relations}. 

The structure of the paper is as follows. Definitions of scale groups and self-replicating groups of automorphisms of regular rooted trees are given in Section~\ref{sec:scalegp}. Descriptions in terms of strings of symbols are given of the regular tree, $\treeq$, in which all vertices have valency $q+1$ and the rooted tree, $\treeqq$, in which all vertices have $q$ children well-known isometry groups are described in terms of these to motivate later ideas. Labellings of $\treeq$ compatible with a scale group are defined in Section~\ref{sec:s.s&s.r}, and the equivalence between scale groups and self-replicating groups is shown. 

Further concepts, results and questions about scale groups complete the paper. The \emph{residue} of a scale group, a group of permutations on $\{0,\dots,q-1\}$ analogous to the residue field of a local field, is defined at the end of Section~\ref{sec:s.s&s.r}. The representation of $H\rtimes \langle\alpha\rangle$ as isometries of $\treeq$ is recalled in Section~\ref{sec:tree_rep} and described concretely in terms of a compatible labelling. Scale groups are characterised up to group isomorphism, and it is shown that the tree representation is unique when the residue permutation group is primitive. Section~\ref{sec:double_transitive} surveys groups of isometries of regular trees that are $2$-transitive on the boundary and how such groups relate to scale groups.

Terminology concerning trees is consistent for the most part with \cite{Serre,Tits}. \emph{Vertices} of a tree will be denoted by bold letters, $\vtx{v}$. A \emph{directed edge} is an ordered pair, $(\vtx{v, w})$, and an \emph{undirected edge} is a pair $\{(\vtx{v, w}),(\vtx{w, v})\}$, which is abbreviated to $\{\vtx{v, w}\}$. The sets of vertices and edges in the tree $\tree$ will be denoted by $\gr{V}\tree$ and $\gr{E}\tree$ respectively. A \emph{path} in a tree is a one-to-one mapping $\vtx{p} : I \to \gr{V}\tree$ with $I$ a finite, semi-infinite or infinite interval in $\mathbb{Z}$ and $(\vtx{p}_n,\vtx{p}_{n+1})$ a directed edge with $\vtx{p}_{n-1}\ne\vtx{p}_{n+1}$ whenever $n-1,n+1\in I$. A path $\vtx{p}$ which  has $I=\mathbb{N}$ is a \emph{ray}, and two rays $\vtx{p},\vtx{p}'$ are \emph{equivalent} if $\vtx{p}(\mathbb{N})\cap \vtx{p}'(\mathbb{N})$ is infinite. An equivalence class of rays is an \emph{end} of the tree $\tree$ and the set of ends is denoted $\partial\tree$. The notation $[\vtx{v},\omega)$, with $\vtx{v}$ a vertex and $\omega$ an end, will mean the ray $\vtx{p}$ such that $\vtx{p}_0 = \vtx{v}$ and $\vtx{p}\in \omega$, and $(\omega,\vtx{v}]$ will mean the path $n\mapsto \vtx{v}_{-n}:\ -\mathbb{N}\to \gr{V}\tree$ with $\vtx{p}$ a ray.

Every automorphism of the combinatorial tree $\tree$ is an isometry of its geometric realisation, and conversely, and the group of all automorphisms will be denoted by $\Isom(\tree)$. For any subgroup $G\leq \Isom(\tree)$ and $\vtx{v}\in V$, the stabiliser of $\vtx{v}$ is $G_{\vtx{v}}$. The action of $G$ on $\tree$ by automorphisms will be denoted $G\curvearrowright \gr{V}\tree$ and the image of $\vtx{v}$ under $g\in G$ by $g.\vtx{v}$. A tree homomorphism $\varphi : \tree^{(1)}\to \tree^{(2)}$ is $G$-\emph{equivariant} if $\varphi(g.\vtx{v}) = g.\vtx{v}$ for all $\vtx{v}\in\gr{V}\tree$ and $g\in G$. 

Throughout, $p$ denotes a prime number, $q$ may be a power of a prime but often is not, and $\mathbb{N}$ includes 0. 

\paragraph{Acknowledgements} I am grateful for valuable discussions with A.~Garrido, R.~Grigorchuk, P.-E.~Caprace, D.~Horadam, V.~Nekrashevych, Z.~\v{S}uni\'{c} and S.~Tornier that assisted the writing of this article and for comments by R.~Grigorchuk correcting an earlier version of the manuscript. I also thank G.~Pisier for hosting an enjoyable visit to Texas A\&M at College Station during which some of these discussions took place. 



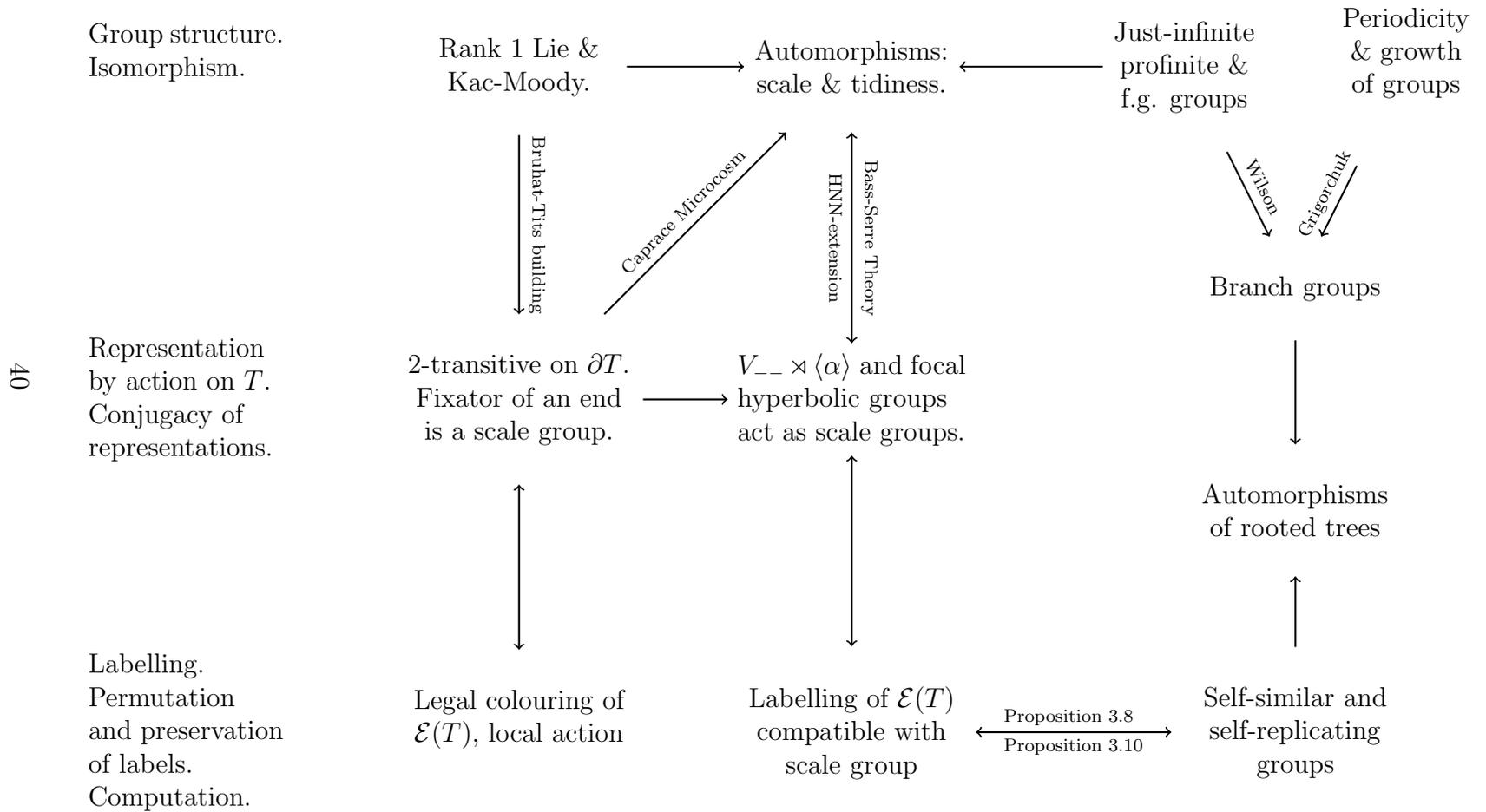
\begin{sidewaysfigure}[h]
\label{fig:concept_relations}
\caption{Scale groups in relation to TDLC groups and self-replicating groups}
\begin{tikzpicture}[node distance = 1.7cm, thick, black]
\node (A)  {\parbox{3cm}{Group structure.\\ Isomorphism.\\ \phantom{XXX}}};
\node (B) [right of = A] {};
\node (BB) [right of = B] {};
\node (C) [right of = BB]  {\parbox{3cm}{\begin{center}Rank 1 Lie \& \\ Kac-Moody.  \end{center}}};
\node (D) [right of = C] {};
\node (DD) [right of = D] {};
\node (E) [right of = DD] {\parbox{3cm}{\begin{center} Automorphisms:\\ scale \& tidiness.  \end{center}}};
\node (F) [right of = E] {};
\node (FF) [right of = F] {};
\node (G) [right of = FF] {\parbox{2.2cm}{\begin{center}Just-infinite profinite \&
		f.g. groups\end{center}}};
\node (H) [right of = G] {};
\node (I) [right of = H] {\parbox{2cm}{\begin{center}Periodicity \\
		\& growth\\ of groups\\ \phantom{XXX}\end{center}}};
\node (2A) [below of=A] {};
\node (22A) [below of=2A] {};
\node (222A) [below of=22A] {\parbox{3cm}{Representation by action on $\tree$.\\ Conjugacy of\\ representations.}};
\node (2B) [right of = 222A] {};
\node (2BB) [right of = 2B] {};
\node (2C) [right of = 2BB]  {\parbox{3.5cm}{\begin{center}2-transitive on $\partial\tree$. \\
		Fixator of an end is a scale group. \end{center}}};
\node (2D) [right of = 2C] {};
\node (2DD) [right of = 2D] {};
\node (2E) [right of = 2DD] {\parbox{3.5cm}{$V_{--}\rtimes\langle\alpha\rangle$ and focal hyperbolic groups\\ act as scale groups.}};
\node (2H) [below of = H] {};
\node (2G) [below of = 2H] {\parbox{3cm}{\begin{center}Branch groups\end{center}}};
\node (22G) [below of = 2G] {};
\node (222G) [below of = 22G] {\parbox{3cm}{\begin{center}Automorphisms of rooted trees\end{center}}};
\node (3A) [below of=222A] {};
\node (33A) [below of=3A] {};
\node (333A) [below of=33A] {\parbox{3cm}{Labelling.\\ Permutation\\ and preservation\\ of labels.\\ Computation.}};
\node (3C) [below of=2C] {};
\node (33C) [below of=3C] {};
\node (333C) [below of=33C] {\parbox{3.5cm}{\begin{center}Legal colouring of $\gr{E}(\tree)$, local action\\\phantom{Fill a lineXXX}  \end{center}}};
\node (3E) [below of=2E] {};
\node (3EE) [below of=3E] {};
\node (3EEE) [below of=3EE] {\parbox{3.5cm}{\begin{center} Labelling of $\gr{E}(\tree)$\\ compatible with scale group\end{center}}};
\node (3F) [right of=3EEE] {};
\node (3FF) [right of=3F] {};
\node (3FFF) [right of=3FF] {};
\node (3G) [right of=3FFF] {\parbox{3.5cm}{\begin{center}Self-similar and self-replicating groups\end{center}}};

\draw[->] (C) -- (E);
\draw[->] (G) -- (E);
\draw[->] (C) -- node[right]{\rotatebox{270}{\scriptsize Bruhat-Tits building}} (2C);
\draw[->] (2C) -- node{\rotatebox{45}{\raisebox{15pt}{\scriptsize Caprace Microcosm}}} (E);
\draw[->] (2C) -- (2E);
\draw[<->] (E) -- node[right]{\rotatebox{270}{\scriptsize Bass-Serre Theory}}node[left]{\rotatebox{270}{\scriptsize HNN-extension}}(2E);
\draw[<->] (2C) -- (333C);
\draw[<->] (2E) -- (3EEE);
\draw[<->] (3EEE) -- node{\raisebox{15pt}{\scriptsize Proposition~\ref{prop:s.s.&s.r.}\phantom{0}}} node{\raisebox{-18pt}{\scriptsize Proposition~\ref{prop:s.s.&s.r.2}}} (3G);
\draw[->] (3G) -- (222G);
\draw[->] (2G) -- (222G);
\draw[->] (G) -- node{\rotatebox{297}{\raisebox{15pt}{\scriptsize Wilson}}}(2G);
\draw[->] (I) -- node{\rotatebox{63}{\raisebox{15pt}{\scriptsize Grigorchuk}}}(2G);
\end{tikzpicture}

\end{sidewaysfigure}



\section{Scale groups and self-replicating groups}
\label{sec:scalegp}

This section defines scale groups and derives their basic properties. What it means for a group to be self-replicating is recalled from~\cite{Nekrash}, and labellings of trees that will be used in later sections are defined. 

 
\subsection{Scale groups}
\label{sec:scale}

\begin{definition}
\label{defn:scalegp}
Let $\treeq$ be a regular tree with valency $q+1$ and choose an end $\omega\in \partial\treeq$. A closed subgroup, $P$, of $\Isom(\treeq)$ that fixes $\omega$ and is transitive on $\gr{V}\treeq$ is called a \emph{scale group}. The set of scale subgroups of $\Isom(\treeq)_\omega$ is denoted $\inftran$. 
\end{definition}

The end $\omega\in\partial\treeq$ will be fixed throughout. The following notation will be useful. For each vertex $\vtx{v}\in \gr{V}\treeq$:
\begin{itemize}
    \item $\gr{E}_\vtx{v}$ is the set of $q+1$ directed edges $\left\{(\vtx{v},\vtx{w})\in \gr{E}\treeq\right\}$; 
\item the vertex adjacent to $\vtx{v}$ on the ray $(\omega,\vtx{v}]$ is denoted by $\vtx{v}^+$; 
\item $\overrightarrow{\gr{N}}(\vtx{v})$ denotes the set of $q$ neighbours of $\vtx{v}$ other than $\vtx{v}^+$; and
\item $\tree_\vtx{v}$ is the component of $\treeq$ formed by deleting the edge $\{\vtx{v},\vtx{v}^+\}$ but keeping $\vtx{v}$. 
\end{itemize}

It is shown in \cite[Proposition~3.2]{Tits} that there are three types of automorphism of a tree: those that fix a vertex; those that interchange two adjacent vertices; and those that translate an infinite path. An automorphism that fixes an end of the tree cannot be of the second type and so, for scale groups, all automorphisms either fix a vertex~$\vtx{v}$, and then fix every vertex on the ray $[\vtx{v},\omega)$, or translate a path that has $\omega$ as either a repelling end or an attracting end. Automorphisms that fix a vertex are called \emph{elliptic} and those that translate a path are called \emph{hyperbolic}. Denote the set of elliptic elements in the group $G$ by $G^{(e)}$.

The following is a well known and easily proved fact about groups of tree automorphisms that fix an end. 
\begin{proposition}
\label{prop:semidirect}
Let $G$ be a closed subgroup of $\Isom(\treeq)$ that fixes the end~$\omega$. Then $G^{(e)}$ is a closed subgroup of~$G$ and either $G = G^{(e)}$ or there is a translation $x\in G$ with $\omega$ as a repelling end such that $G = G^{(e)}\rtimes \langle x\rangle$. \endproof
\end{proposition}

Scale groups always contain a translation, $x$, by distance~$1$ and have the form $P =  P^{(e)}\rtimes\langle  x\rangle$, as the next results show.
\begin{proposition}
\label{prop:scale_gp_props}
Let $P\leq \Isom(\treeq)_\omega$ be a scale group and $\vtx{v}\in\gr{V}\treeq$. Suppose that $V = P_{\vtx{v}}$. Then there is $x\in P$ such that $xVx^{-1}<V$ and $[V : xVx^{-1}]=q$. For any such $x$, $P^{(e)} = \bigcup_{n\in\mathbb{Z}} x^nVx^{-n}$ and $P =  P^{(e)}\rtimes\langle  x\rangle$.
\end{proposition}
\begin{proof}
Let $\vtx{w}\in \overrightarrow{\gr{N}}(\vtx{v})$ and let $x\in P$ be such that $x.\vtx{v} = \vtx{w}$. Then $xVx^{-1} = P_{\vtx{w}}$, which is a proper subgroup of $V$ because $P_{\vtx{w}}$ fixes every vertex on the ray $[\vtx{w},\omega)$. Moreover, $V.\vtx{w}$ is contained in $\overrightarrow{\gr{N}}(\vtx{v})$, and is equal to this set because, for each $\vtx{u}\in \overrightarrow{\gr{N}}(\vtx{v})$, there is $y\in P$ with $y.\vtx{w}=\vtx{u}$ and any such $y$ must belong to $V$. Hence, by the Orbit-Stabiliser Theorem, $[V : xVx^{-1}]=|\overrightarrow{\gr{N}}(\vtx{v})| = q$.

For the second claim, consider $h\in P^{(e)}$. Then $h$ fixes a ray, $[\vtx{u},\omega)$ and $[\vtx{u},\omega)\cap [\vtx{v},\omega)$ is a ray, $[\vtx{v}',\omega)$ say. Since $\vtx{v}'\in [\vtx{v},\omega)$, there is $m\in\mathbb{Z}$ such that $x^{-m}.\vtx{v} = \vtx{v}'$. Then $x^{-m}Vx^m = P_{\vtx{v}'}$, to which $h$ belongs. 
\end{proof}

The next result applies to semiregular trees, that is, trees whose vertex set partitions into two subsets, $\gr{V}_1$ and $\gr{V}_2$, such that all vertices in each subset have the same number of neighbours, all belonging to the other subset. Although only needed for regular trees here, the greater generality might be useful in other contexts. It may be compared with~\cite[Lemma~2.2]{Radu}.
\begin{proposition}
\label{prop:transitive_on_ends}
Let $\tree$ be a semiregular tree and suppose that $G\leq \Isom(\tree)$ is closed, transitive on each of the vertex sets $\gr{V}_1,\ \gr{V}_2$ with $\gr{V}_1\cup\gr{V}_2=\gr{V}\tree$, and fixes an end, $\omega\in\partial\tree$. Then $G^{(e)}$ is transitive on $\partial\tree\setminus\{\omega\}$. 

This includes the case of regular trees on which $G$ is vertex-transitive and so, in particular, $P^{(e)}$ acts transitively on $\partial\treeq\setminus\{\omega\}$ for every $P$ in $\inftran$.  
\end{proposition}
\begin{proof}
Let $\omega_1\ne\omega_2\in \partial\treeq\setminus\{\omega\}$. Then $(\omega_1,\omega)\cap (\omega_2,\omega) = [\vtx{v},\omega)$ for some vertex $\vtx{v}\in \gr{V}\treeq$. For each $n\geq1$ let $\vtx{v}_n$ and $\vtx{w}_n$ be the vertices on $(\omega_1,\vtx{v}]$ and $(\omega_2,\vtx{v}]$ respectively at distance $n$ from $\vtx{v}$. Then, since $\vtx{v}_n$ and $\vtx{w}_n$ are equidistant from $\vtx{v}$, they lie in the same set, $\gr{V}_1$ or $\gr{V}_2$, of vertices. Transitivity of the action of $G$ on $\gr{V}_1$ and $\gr{V}_2$ then implies that there is, for each $n$, $x_n\in G$ such that $x_n.\vtx{v}_n = \vtx{w}_n$. 

Since each $x_n$ fixes $\omega$, it follows that $x_n.\vtx{v}=\vtx{v}$ for every $n$ and, moreover, that $x_n.\vtx{v}_j=\vtx{w}_j$ for every $1\leq j\leq n$. Hence the set $\{x_n\}_{n\geq1}$ is contained in the stabiliser, $G_{\vtx{v}}$, of $\vtx{v}$, which is compact, and therefore has an accumulation point, $x$ say. Then $x.\vtx{v}_n = \vtx{w}_n$ for every $n\geq1$ and it follows that $x.\omega_1 = \omega_2$. Since $x$ fixes $\vtx{v}$ it is elliptic, and we have shown that the elliptic subgroup of $G$ is transitive on $\partial\treeq$.
\end{proof}

\begin{corollary}
\label{cor:transitive_on_ends}
Suppose that $\omega,\omega'\in\partial \treeq$ and that $\{\vtx{v},\vtx{w}\}\in \gr{E}\treeq$ lies on the path $(\omega,\omega')$. Then, given a scale group $P\leq \Isom(\treeq)_\omega$, there is $x\in P$ satisfying $x.\omega = \omega$, $x.\omega' = \omega'$ and $x.\vtx{v} = \vtx{w}$.
\end{corollary}
\begin{proof}
There is $y\in P$ with $y.\omega = \omega$ and $y.\vtx{v}=\vtx{w}$ because $P$ is a scale group. Then the proof of Proposition~\ref{prop:transitive_on_ends} shows that there is $z\in P$ with $z.(y.\omega') = \omega'$ and $z.\vtx{w}=\vtx{w}$. Set $x = zy$.
\end{proof}

Proposition~\ref{prop:transitive_on_ends} has the following converse.
\begin{proposition}
\label{prop:implies_scale}
Let $G\leq \Isom(\tree)_\omega$ be closed. Suppose that $G^{(e)}$ is transitive on $\partial\tree\setminus\{\omega\}$ and that there are $\vtx{v}\in\gr{V}\treeq$ and $g\in G$ such that $\{\vtx{v},g.\vtx{v}\}\in\gr{E}\treeq$. Then $G$ is a scale group.
\end{proposition}
\begin{proof}
It must shown that $G$ is transitive on $\gr{V}\treeq$. Let $\vtx{v}\in\gr{V}\treeq$ and $g\in G$ be such that $\{\vtx{v},g.\vtx{v}\}\in\gr{E}\treeq$. Then $g$ is hyperbolic and there is $\omega_1\in\partial\tree\setminus\{\omega\}$ with $(\omega,\omega_1)$ the axis of translation for $g$. Consider $\vtx{w}\in\gr{V}\treeq$, let $\omega_2\in\partial\tree\setminus\{\omega\}$ be such that $\vtx{w}\in(\omega,\omega_2)$ and choose $h\in G^{(e)}$ such that $h.\omega_1 = \omega_2$. Then $h^{-1}.\vtx{w}\in (\omega,\omega_1)$ and there is $n\in\mathbb{Z}$ such that $g^n.\vtx{v} = h^{-1}.\vtx{w}$. Hence $\vtx{w} = (hg^n).\vtx{v}$ and $G$ is transitive on $\gr{V}\treeq$.
\end{proof}

\begin{example}
\label{examp:scale_gp}
The $(ax+b)$-group $\mathbb{Q}_p\rtimes \mathbb{Q}_p^\times$ acts on $\tree_{p+1}$ via the embedding $\mathbb{Q}_p\rtimes \mathbb{Q}_p^\times\hookrightarrow GL(2,\mathbb{Q}_p)$ given by
$$
(b,a) \mapsto \left(\begin{matrix}
a & b\\
0 & 1
\end{matrix}\right),\quad (a\in\mathbb{Q}_p^\times,\ b\in\mathbb{Q}_p)
$$
and the action of $GL(2,\mathbb{Q}_p)$ on its affine Bruhat-Tits tree~\cite[Chapter~II.1]{Serre}. The boundary of the Bruhat-Tits tree may be identified with the projective line $P^1(\mathbb{Q}_p)$ and $\mathbb{Q}_p\rtimes \mathbb{Q}_p^\times$ fixes the point with homogeneous coordinates $[1:0]$ under its action on the tree. Let $P\leq \Isom(\tree_{p+1})_{[1:0]}$ be the image of $\mathbb{Q}_p\rtimes \mathbb{Q}_p^\times$ under this action. Then $P$ is a scale group by Proposition~\ref{prop:implies_scale}: that $P$ is transitive on $\partial\tree_{p+1}\setminus \{[1:0]\}$ holds because we have that $[b,1] = (b,1).[0:1]$ for each $[b,1]\in\partial\tree_{p+1}\setminus \{[1:0]\}$; and $(0,p)$ is hyperbolic with translation distance~$1$.  

The group $\mathbb{F}_q(\!(t)\!)\rtimes \mathbb{F}_q(\!(t)\!)^\times$ acts on $\treeq$ in a similar fashion and its image under the action is another example of a scale group. 
\end{example}
There is an alternative, and well-known, construction of the action of $\mathbb{Q}_p\rtimes \mathbb{Q}_p^\times$ on a tree isomorphic to the Bruhat-Tits tree, and of the identification of the boundary of the tree with $P^1(\mathbb{Q}_p)$. This construction is given in Example~\ref{examp:padic} below, and is the model for a concrete description of $\treeq$ and the action of $\Isom(\treeq)_\omega$ that is the subject of \S\ref{sec:concrete}. Example~\ref{examp:scale_group} defines a particular class of scale groups in  terms of this construction of the tree that is inspired by the example of $\mathbb{F}_q(\!(t)\!)\rtimes \mathbb{F}_q(\!(t)\!)^\times$.


\subsection{Self-replicating groups}
\label{sec:s.r.}
Self-similar and self-replicating groups are groups of automorphisms of regular rooted trees, and their definitions involve representing the vertices of these trees as strings in an `alphabet', see~\cite[\S1.1]{Nekrash}. The alphabet is here taken to be the set of $q$ symbols $\{0,\dots, q-1\}$. The following definition and notation for the rooted tree $\treeqq$ do not match that in~\cite{Nekrash} exactly but are better suited to our context. 

\begin{numba}
\label{nota:treeqq}
Let $\gr{V} = \{0,\dots, q-1\}^*$ be the set of finite strings of symbols $0,\ \dots,\ q-1$ and $\gr{E} = \left\{ \{\vtx{w}, \vtx{w}j\}\in \gr{V}^2 \mid \vtx{w}\in \gr{V}, j\in \{0,\dots, q-1\}\right\}$. Then $\treeqq := (\gr{V},\gr{E})$ is a rooted tree whose root is the empty string and in which the vertex $\vtx{w}$ has the~$q$ children $\vtx{w}j$, $j\in\{0,\dots q-1\}$. Strings of length $n$ are said to be on \emph{level~$n$} of $\treeqq$. 
\end{numba}

\begin{numba}
\label{num:rooted_section}	For each $\vtx{w}\in \gr{V}\tree_{q,q}$, the subtree spanned by $\left\{\vtx{w}\vtx{v}\in  \gr{V}\tree_{q,q} \mid \vtx{v} \in \gr{V}\tree_{q,q}\right\}$ will be denoted by $\tree_\vtx{w}$. Then $\tree_\vtx{w}$ is isomorphic to $\tree_{q,q,}$ under the map
$$
\psi_\vtx{w}(\vtx{w}\vtx{v}) = \vtx{v} : \tree_\vtx{w}\to\tree_{q,q}.
$$ 
For any $\vtx{w}\in\gr{V}\treeq$, the map $\cdot|_{\vtx{w}} :  \Isom(\tree_{q,q}) \to \Isom(\tree_{q,q})$ defined by
$$
g|_{\vtx{w}} = \psi_{g.\vtx{w}}\circ g\circ \psi_\vtx{w}^{-1}
$$ 
is the \emph{section map} at $\vtx{w}$, see~\cite[\S1.3]{Nekrash} where $\cdot|_{\vtx{w}}$ is called \emph{restriction}. 
\end{numba}

The set of semi-infinite strings $\{0,\ \dots,\ q-1\}^{\mathbb{N}}$ may be identified with the set of ends, $\partial\treeqq$, of the rooted tree.

The next definitions may be found in~\cite[Definitions~1.5.3 and~2.8.1]{Nekrash}, although self-replicating groups are called \emph{recurrent} in that book and given the present name only in~\cite{Nekrash2} (also see~\cite{GriRev}).
\begin{definition}
\label{defn:s.s.&s.r.}
The subgroup, $G$, of $\Isom(\tree_{q,q})$ is \emph{self-similar} if, for each $\vtx{w}\in \gr{V}\treeqq$, the section $g|_\vtx{w}$ belongs to $G$  for every $g\in G$.

The self-similar group, $G$, is \emph{self-replicating} if it is transitive on level~1 of $\treeqq$ and the section map 
$|_{\vtx{w}} : G_\vtx{w}\to G$ is surjective for every $\vtx{w}\in \gr{V}\treeqq$. 
\end{definition}
\paragraph{Note:}an earlier version of this article gave a weaker definition of self-similar groups which required that $g|_\vtx{w}$ belongs to $G$ only for $g\in G_\vtx{w}$. The two definitions are compared in \cite[\S2.4]{GriSav}.

The examples of self-similar groups given in~\cite[\S1.6--9]{Nekrash} include many examples of self-replicating groups such as the Grigorchuk~\cite{Grigorchuk} and Gupta-Sidki~\cite{Gupta-Sidki} groups and the odometer action of $\mathbb{Z}$ on $\treeqq$. That $G$ should be closed in $\Isom(\treeqq)$ is not part of the definition that $G$ be self-replicating. The groups considered in this article are closed, however, and the closure of a self-replicating group, such as one of those described in~\cite{Nekrash}, is self-replicating. 

\begin{numba}
\label{nota:treeqqo}
The notation $\treeqqo$ will be used for the rooted tree in which the root has $q-1$ children and every other vertex has $q$ children. This will be identified with the rooted subtree of $\treeqq$ spanned by finite strings in $\{0, \dots, q-1\}$ whose first letter is not $0$.
\end{numba}

\subsection{The regular tree in terms of sequences of labels}
\label{sec:concrete}

This subsection describes the regular tree, $\treeq$, in terms of infinite strings of symbols in an analogous way to the description of $\treeqq$ in paragraph~\ref{nota:treeqq}. The labellings of vertices and edges implicit in this description underpin the connection between scale groups and self-replicating groups shown Section~\ref{sec:s.s&s.r}. The description given here has also been given in \cite[\S2.2]{B_T_Zheng}.

The end, $\omega$, fixed by a scale group plays a distinguished role like that played by the root vertex in the case of self-replicating groups. Corresponding to the levels of a rooted tree are levels relative to $\omega$ in a regular tree. 
\begin{definition}
\label{defn:horo} 
Fix $\omega\in\partial\treeq$, choose a vertex $\vtx{v}_0\in \gr{V}\treeq$ and define $B_{\vtx{v}_0} : \gr{V}\treeq\to\mathbb{Z}$ by 
$$
B_{\vtx{v}_0}(\vtx{v}) = d(\vtx{v},\vtx{w}) - d(\vtx{v}_0,\vtx{w})\quad (\vtx{w}\in [\vtx{v},\omega)\cap [\vtx{v}_0,\omega)).
$$
(The number on the right is independent of $\vtx{w}\in [\vtx{v},\omega)\cap [\vtx{v}_0,\omega)$). The function $B_{\vtx{v}_0}$ is a \emph{Busemann function} for $\omega$ and the level sets of this function are called \emph{horospheres}. 
\end{definition}
See \cite[Chapter~II.8]{BridHaef} for Busemann functions and horospheres. The horospheres in $\treeq$ depend only on $\omega$ and not on the choice $\vtx{v}_0$. Horospheres are counterparts of the levels of a rooted tree.

The tree $\treeq$ may be defined as the span of the horospheres for a specific end. The same definition is given in \cite[\S2.2]{B_T_Zheng}.
\begin{definition}
\label{defn:treeq+1}
For each $n\in\mathbb{Z}$, define
$$
\gr{V}_n:= \left\{ \vtx{w} = (w_i) \in \{0,\dots,q-1\}^{(-\infty,n]} \mid w_i=0\mbox{ for almost all }i \right\}
$$
and
$$
\gr{V} := \bigcup_{n\in\mathbb{Z}} \gr{V}_n.
$$
Thus $\gr{V}$ is a set of semi-infinite strings in the symbols $\{0,\dots,q-1\}$ with all but finitely many symbols in each string being equal to~$0$. 

For each $\vtx{w}\in \gr{V}_n$, $\vtx{w}^+ = \vtx{w}|_{(-\infty,n-1]}$ is the string which omits the last term of $\vtx{w}$ and, for each $j\in\{0,\dots,q-1\}$, $\vtx{w}j\in \gr{V}_{n+1}$ is defined by
$$
(\vtx{w}j)_i = w_i\mbox{ if }i\leq n \mbox{ and }(\vtx{w}j)_{n+1} = j.
$$
Then $\vtx{w}\mapsto \vtx{w}^+ : \gr{V}_n\to \gr{V}_{n-1}$ is a $q$-to-$1$ surjective map. Define 
$$
\gr{E} = \left\{ \{\vtx{w},\vtx{w}^+\} \mid \vtx{w}\in \gr{V}\right\}.
$$
Then each $\vtx{w}\in \gr{V}$ has the $q+1$ neighbours 
$$
\{\vtx{w}^+\}\cup \left\{ \vtx{v}\in \gr{V} \mid \vtx{v}^+=\vtx{w}\right\} = \{\vtx{w}^+\}\cup \left\{ \vtx{w}j\in \gr{V} \mid j\in\{0,\dots,q-1\}\right\}.
$$ 
The $(q+1)$-regular tree $\wtreeq$ will be henceforth identified with the graph $(\gr{V},\gr{E})$.

For $\vtx{w}\in\gr{V}_n$ and $j\in \mathbb{N}$, $\vtx{w}|_{(-\infty,n-j]} \in \gr{V}_{n-j}$ is the semi-infinite string which omits the last $j$ terms in $\vtx{w}$. Then $\{\vtx{w}|_{(-\infty,n-j]}\}_{j\in \mathbb{N}}$ is a ray in $\wtreeq$. Denote the string in $\gr{V}_n$ in which all symbols equal~$0$ by $\tvtx{v}_n$. Then for each $\vtx{w}\in\gr{V}$ there is $J$ such that $\vtx{w}|_{(-\infty,n-j]} = \tvtx{v}_j$ for all $j>J$ and hence the rays $\{\vtx{w}|_{(-\infty,n-j]}\}_{j\leq n}$, $\vtx{w}\in \gr{V}_n\subset\gr{V}$, all lie in the same end. This end will be denoted by~$\omega$. 

Define the particular element $\tilde{x}_0\in \Isom(\wtreeq)_\omega$ by
\begin{equation}
\label{eq:wxzero}
(\tilde{x}_0.\vtx{w})_m = w_{m-1},  \quad (\vtx{w}=(w_m)\in \gr{V}\wtreeq).
\end{equation}
In particular, $\tilde{x}_0.\tvtx{v}_n = \tvtx{v}_{n+1}$, $(n\in\mathbb{Z})$.
\end{definition}

\begin{remark}
	\label{rem:ends&sides}
	The sets $\gr{V}_n$ are the horospheres for $\omega$. The element $\tilde{x}_0$ is a translation with repelling end $\omega$ and axis of translation the path $\{\tvtx{v}_n\}_{n\in\mathbb{Z}}$. 
	
	Given a bi-infinite sequence $(w_i)_{i\in\mathbb{Z}}$ with $w_i\in \{0,\dots,q-1\}$ and $w_i=0$ as $i\to-\infty$, put $\vtx{w}_n = (w_i)_{i\leq n}\in \{0,\dots,q-1\}^{(-\infty,n]}$. Then $\{\vtx{w}_n\}_{n\in\mathbb{Z}}$ is a path in $\wtreeq$. The ray $\{\vtx{w}_n\}_{n\leq0}$ belongs to the end $\omega$ and the ray $\{\vtx{w}_n\}_{n\geq0}$ determines a unique end of $\treeq$ that is not equal to $\omega$. The bi-infinite sequences $(w_i)_{i\in\mathbb{Z}}$  thus correspond to $\partial\wtreeq\setminus\{\omega\}$. 
\end{remark}

\begin{numba}
	Every $\vtx{v}\in\gr{V}\treeq$ has the form $\vtx{v} = \tvtx{v}_n\vtx{w}$ with $n\in\mathbb{Z}$ and $\vtx{w}\in\treeqqo$ unique. (Recall from paragraph~\ref{nota:treeqqo} that strings in $\gr{V}\treeqqo$ have their first letter not~$0$.) Then the element $\tilde{x}_0\in\Isom(\treeq)$ satisfies 
	$$
	\tilde{x}_0.\vtx{v} = \tilde{x}_0.(\tvtx{v}_n\vtx{w}) = \tvtx{v}_{n+1}\vtx{w},\quad (\vtx{v}\in \gr{V}\treeq).
	$$
\end{numba}

\begin{numba}
	\label{num:section_regular}
	Let $\vtx{w}$ be in the horosphere $\gr{V}_n$ of $\treeq$. Then each $\vtx{v}\in \gr{V}\tree_{\vtx{w}}$ belongs to a horosphere $\gr{V}_m$ with $m\geq n$ and has the form $\vtx{w}v_{n+1}\dots v_{m}$ with $v_j\in q$. As in paragraph~\ref{num:rooted_section}, define the isomorphism $\phi_{\vtx{w}}: \tree_{\vtx{w}} \to \treeqq$ by
\begin{equation*}
\label{eq:rooted_subtrees}
\phi_{\vtx{w}} : \vtx{w}v_{n+1}\dots v_{m}\mapsto v_{n+1}\dots v_{m}.
\end{equation*}
Then, for each $g\in \Isom(\treeq)$,  the \emph{section} of $g$ at $\vtx{w}$ is  the composite map 
\begin{equation*}
\label{eq:section}
g|_{\vtx{w}} = \phi_{g.\vtx{w}}\circ g\circ \phi_\vtx{w}^{-1}\in \Isom(\tree_{q,q}).
\end{equation*}
\end{numba}

 Definition~\ref{defn:treeq+1} facilitates a direct description of the representation of $\mathbb{Q}_p\rtimes \mathbb{Q}_p^\times$ as a scale group on the Bruhat-Tits tree seen in Example~\ref{examp:scale_gp}.
\begin{example}
	\label{examp:padic}
	Construct a graph by defining its vertex and edge sets to be 
	\begin{align*}
	\gr{V} &= \left\{ y + p^{n+1}\mathbb{Z}_p \mid n\in\mathbb{Z},\ y\in\mathbb{Q}_p\right\}
	\\
	\mbox{ and }\ \gr{E} &= \left\{ (y + p^{n}\mathbb{Z}_p,y + p^{n+1}\mathbb{Z}_p) \mid n\in\mathbb{Z},\ y\in\mathbb{Q}_p\right\}.
	\end{align*}
	Each $y\in \mathbb{Q}_p$ may be written uniquely as $y = \sum_{n=-\infty}^\infty y_n p^n$ with the coefficients $y_n\in\{0,\dots,p-1\}$ for each $n$ and $y_n=0$ as $n\to-\infty$. Hence the vertex $y + p^{n+1}\mathbb{Z}_p$ may be specified by the sequence $(y_m)_{m\leq n} \in \{0,\dots,p-1\}^{(-\infty,n]}$. This vertex then has the $p+1$ neighbours $y + p^{n}\mathbb{Z}_p$ and $y + p^n y_{n+1}+ p^{n+2}\mathbb{Z}_p$, $y_{n+1}\in\{0,\dots,p-1\}$ and the graph is isomorphic to $\tree_{p+1}$.  
	
	The $(ax+b)$-group $\mathbb{Q}_p\rtimes \mathbb{Q}_p^\times$ acts on $\gr{V}$ by
	\begin{equation}
	\label{eq:padic_action}
	(b,a).(y + p^{n+1}\mathbb{Z}_p) = (b+ay) + ap^{n+1}\mathbb{Z}_p,\quad (y + p^{n+1}\mathbb{Z}_p\in \gr{V}).
	\end{equation}
	This action preserves the edge relations and thus defines an action by automorphisms of $\tree_{p+1}$. 
	
	For each fixed $y\in\mathbb{Q}_p$ the bi-infinite sequence $\{y+ p^{n+1}\mathbb{Z}_p\}_{n\in\mathbb{Z}}$ is a path in the tree. All rays $\{y+ p^{n+1}\mathbb{Z}_p\}_{n\leq 0}$ ($y\in\mathbb{Q}_p$) belong to the same end of the tree because, for every $y$, there is $m\in\mathbb{Z}$ such that $y_{m'}=0$ if $m'\leq m$. Identify this end with $[1:0]\in P^1(\mathbb{Q}_p)$ and identify the end containing the ray $\{y+ p^{n+1}\mathbb{Z}_p\}_{n\geq 0}$ with $[y:1]$. The boundary of the tree is thus identified with the projective line $P^1(\mathbb{Q}_p) = \{[1:0]\}\cup \{[y:1] \mid y\in\mathbb{Q}_p\}$. 
	The horospheres for the end $[1:0]$ are the sets $\gr{V}_n := \mathbb{Q}_p/p^{n+1}\mathbb{Z}_p$.
\end{example}


The explicit description of $\wtreeq$ in Definition~\ref{defn:treeq+1} allows isometries fixing the end $\omega$ to be presented in terms of finite permutations located at the vertices of the tree. This corresponds to the portrait description of elements of $\Isom(\treeqq)$ in~\cite[\S2.1]{GriSav} and~\cite[\S1.3.2]{Nekrash}.
\begin{proposition}
\label{prop:describe_autos}
For every $x\in \Isom(\wtreeq)_\omega$ there are a unique $n\in\mathbb{Z}$ and elliptic $h\in \Isom(\wtreeq)_\omega$ such that $x = h\tilde{x}_0^n$. 

For the elliptic $h$ and each $\vtx{w} \in \gr{V}\wtreeq$, there is $\pi^{(h)}_\vtx{w}\in \Sym(q)$ such that
\begin{equation}
\label{eq:auto_form0}
h.(\vtx{w}j) = (h.\vtx{w})\pi^{(h)}_\vtx{w}(j), \quad (j\in \{0,\dots, q-1\})
\end{equation}
and there is $N\in\mathbb{Z}$ depending on $h$ such that $\pi^{(h)}_{\tvtx{v}_n}(0) = 0$ for all $n\leq N$. 

Conversely, given a function $\pi : \gr{V}\wtreeq \to \Sym(q)$ and $N\in\mathbb{Z}$ such that $\pi_{\tvtx{v}_n}(0) = 0$ for all $n\leq N$, define, for each $\vtx{w} \in \gr{V}_n\subset \gr{V}\wtreeq$,
\begin{equation}
\label{eq:auto_form2}
w'_i = \pi_{\vtx{w}|_{(-\infty,i-1]}}(w_{i}),\quad (i\in (-\infty,n]).
\end{equation}
Then $\vtx{w}' := (w'_i)_{i\leq n}\in \gr{V}_n$ and the map $\vtx{w}\mapsto \vtx{w}' : \gr{V}\wtreeq \to \gr{V}\wtreeq$ is an elliptic element, $h$, in $\Isom(\wtreeq)_\omega$. 
\end{proposition}
\begin{proof}
Given $x\in \Isom(\wtreeq)_\omega$, there is $n\in \mathbb{Z}$ such that $x.\gr{V}_m = \gr{V}_{m+n}$ for every horosphere $\gr{V}_m$. Then $h := x\tilde{x}_0^{-n}$ is elliptic, thus proving the first claim. 

Next, consider $h$ elliptic, $\vtx{w}\in \gr{V}\wtreeq$ and $i\in\{0,\dots,q-1\}$. Since $(\vtx{w}i)^+ = \vtx{w}$, we have that $(h.(\vtx{w}i))^+ = h.\vtx{w}$ and $h.(\vtx{w}i) = (h.\vtx{w})k$ with $k\in q$. The map $\pi^{(h)}_\vtx{w} : i\mapsto k$ is a bijection because $h$ is, and so $\pi^{(h)}_\vtx{w}$ belongs to $\Sym(q)$ and~\eqref{eq:auto_form0} holds. Moreover, since $h.\omega = \omega$, there is $N\in \mathbb{Z}$ such that $h.\vtx{v}_n = \vtx{v}_{n}$ for all $n\leq N$. Hence $\pi^{(h)}_{\vtx{v}_n}(0) = 0$ for all $n\leq N$. 


Next, suppose that $(w'_i)_{i\leq n}$ is defined by~\eqref{eq:auto_form2}. Choose $N\in\mathbb{Z}$ such that $w_i=0$ and $\pi_{\vtx{v}_i}(0)=0$ for all $i\leq N$. Then $w'_i = 0$ for all $i\leq N$ and $\vtx{w}' = (w'_i)_{i\leq n}\in \gr{V}_n\subset\gr{V}\wtreeq$. For $\vtx{w}\in \gr{V}_n$, both $\vtx{w}^+$ and $h.\vtx{w}^+$  belong to $\gr{V}_{n-1}$ and we have $(h.\vtx{w}^+)_i = (h.\vtx{w})^+_i$ for all $i\in(-\infty,n-1]$.
Hence $h.\vtx{w}^+ = (h.\vtx{w})^+$ and $h$ preserves the edges in $\wtreeq$. With $N$ chosen as above, we have that $h.\tvtx{v}_n = \tvtx{v}_n$ for all $n\leq N$. Hence $h.\omega = \omega$ and $h\in\Isom(\wtreeq)_{\omega}$.
\end{proof}

Proposition~\ref{prop:describe_autos} facilitates the definition of classes of examples of scale groups. The first examples are analogues of the ``universal group" for a local action defined in~\cite{BM}. Let $\pi^{(x)} : \gr{V}\treeq\to \Sym(q)$ be as in the proposition.
\begin{example}
\label{examp:scale_group2}
Let $S\leq \Sym(q)$ be a transitive permutation group. Define
$$
P^{(S)} = \left\{ h\tilde{x}_0^n\in \Isom(\wtreeq)_\omega \mid n\in\mathbb{Z},\ \pi^{(h)}_\vtx{w}\in S\mbox{ for all }\vtx{w}\in\gr{V}\treeq\right\}.
$$
Then $P^{(S)}$ is a closed subgroup of $\Isom(\wtreeq)_\omega$ and is transitive on $\gr{V}\treeq$ because transitivity of $S$ implies that $P^{(S)}_{\vtx{v}}$ is transitive on the children of $\vtx{v}$ and $\tilde{x}_0$ translates the horospheres. In fact, $P^{(\Sym(q))} = \Isom(\treeq)_\omega$.
\end{example}

\begin{example}
\label{examp:scale_group}
Let $S\leq \Sym(q)$ be a transitive permutation group. Define
$$
C^{(S)} = \left\{ h\tilde{x}_0^n\in P^{(S)} \mid \pi^{(h)}\mbox{ is constant on horospheres}\right\}.
$$
In other words, there are $\chi\in S^{\mathbb{Z}}$ and $N\in\mathbb{Z}$ with $\chi_n(0) = 0$ for all $n\leq N$ and $\pi^{(h)}_\vtx{w} = \chi_n$ for all $\vtx{w}\in \gr{V}_n$. Then the same argument as in the previous example shows that $C^{(S)}$ is a scale group.
\end{example}
A particular case of Example~\ref{examp:scale_group} has $S$ the additive group of the field $\mathbb{F}_p$ of order~$p$ acting on itself by translation. Then $C^{(\mathbb{F}_p,+)} \cong \mathbb{F}_p(\!(t)\!)\rtimes \langle t\rangle$.

\section{Edge labellings and scale groups, correspondences with self-replicating groups}
\label{sec:s.s&s.r}

By an \emph{edge-labelling} of $\treeq$ we shall mean a map that assigns a label in $\{0,\dots,q\}$ to each directed edge in $\treeq$ and which is, for every vertex~$\vtx{v}$, a bijection on the set of out-edges from~$\vtx{v}$. Edge labellings compatible with a scale group~$P$, see Definition~\ref{defn:compatible_colour}, are the key to showing the correspondence between scale groups 
and self-replicating groups 
because tree labellings are used in the definition of the latter. This correspondence was shown previously in~\cite{Horadam} using a method that is less direct than compatible labellings. 

Although the idea is inspired by the notion of legal colouring of edges defined in~\cite{BM}, edge labellings compatible with~$P$ differ in that: all edges pointing towards the end~$\omega$ are labelled~$q$, see Definition~\ref{defn:compatible_colour}.\ref{defn:compatible_colour1}; the label $0$ also has a special role, see Definition~\ref{defn:compatible_colour}.\ref{defn:compatible_colour2}; and compatibility with~$P$ means that there are elements of $P$ that preserve the labelling on half-trees not containing~$\omega$, see Definition~\ref{defn:compatible_colour}.\ref{defn:compatible_colour3}. 

\subsection{Edge labellings compatible with a scale group}
\label{sec:edgelabelling}

Recall from \S\ref{defn:scalegp} that: $\omega$ is a fixed end of $\treeq$; $\gr{E}_\vtx{v}$ is the set of out-edges at~$\vtx{v}$; $\vtx{v}^+$ is the unique neighbour of $\vtx{v}$ that is closer to $\omega$; and $\tree_\vtx{v}$ is the subtree of $\treeq$ spanned by the set of vertices~$\vtx{w}$ such that $\vtx{v}$ lies on the ray $(\omega,\vtx{w}]$.  

\begin{definition}
\label{defn:compatible_colour}
Suppose that $P\in \inftran$. A \emph{$P$-labelling of $\gr{E}\treeq$} is a set of bijective maps $c_\vtx{v} : \gr{E}_\vtx{v}\to \{0,1,\dots, q\}$, one for each $\vtx{v}\in \gr{V}\treeq$, such that 
\begin{enumerate}
\item $c_\vtx{v}(\vtx{v},\vtx{v}^+) = q$ for all $\vtx{v}\in \gr{V}\treeq$; 
\label{defn:compatible_colour1}
\item there is a path, $\{\vtx{v}_n\}_{n\in\mathbb{Z}}$, in $\treeq$ with $\vtx{v}_n = \vtx{v}_{n+1}^+$ and $c_{\vtx{v}_n}(\vtx{v}_n,\vtx{v}_{n+1}) = 0$ for all $n\in\mathbb{Z}$; and 
\label{defn:compatible_colour2}
\item for every $\vtx{u}_1,\vtx{u}_2\in \gr{V}\treeq$ there is $x\in P$ such that $x.\vtx{u}_1 = \vtx{u}_2$ and $x|_{\tree_{\vtx{u}_1}}$ preserves the labelling $T_{\vtx{u}_1}\to T_{\vtx{u}_2}$.
\label{defn:compatible_colour3}
\end{enumerate}
A $P$-labelling of $\gr{E}\treeq$ will be said to be \emph{compatible with $P$}. 
\end{definition}
Given a $P$-labelling $\{c_\vtx{v}\}_{\vtx{v}\in \gr{V}\treeq}$ of $\gr{E}\treeq$ and an edge $(\vtx{v},\vtx{w})$, Definition~\ref{defn:compatible_colour}.\ref{defn:compatible_colour1} implies that $c_\vtx{v}(\vtx{v},\vtx{w})\in\{0,\dots,q-1\}$ if and only if $c_\vtx{w}(\vtx{w},\vtx{v}) = q$. 

For every scale group $P$ there is at least one $P$-labelling of $\gr{E}\treeq$.
\begin{proposition}
\label{prop:compatible_label}
Let $P\in \inftran$. Then there is a $P$-labelling of $\gr{E}\treeq$.  
\end{proposition}
\begin{proof}
	Since $P$ is a scale group, there is $x_0\in P$ which translates an axis $\{\vtx{v}_n\}_{n\in\mathbb{Z}}$ and for which $\omega$ is a repelling end. The labelling will be defined initially on the rooted tree $\tree_{\vtx{v}_0}$. Let $\vtx{w}_i$, $i\in\{0,\dots, q-1\}$, be the children of $\vtx{v}_0$ in $\tree_{\vtx{v}_0}$ and, using that $P$ acts transitively on $\gr{V}\treeq$, choose $x_i\in P$ for each~$i\ne0$ such that $x_i.\vtx{v}_0 = \vtx{w}_i$. Since $x_i.\tree_{\vtx{v}_0} = \tree_{\vtx{w}_i}$ for each $i$, there is, for each $\vtx{w}\in \gr{V}\tree_{\vtx{v}_0}$, a unique string $i_1\dots i_d$ in $\{0,\dots,q-1\}^*$ such that $(x_{i_1}\dots x_{i_d}).\vtx{v}_0 = \vtx{w}$ and: 
	\begin{itemize}
		\item $d$ is the level of $\vtx{w}$ in $\tree_{\vtx{v}_0}$;  
		\item $\vtx{w}^+ = (x_{i_1}\dots x_{i_{d-1}}).\vtx{v}_0$; and 
		\item the children of $\vtx{w}$ are $(x_{i_1}\dots x_{i_d}).\vtx{w}_i = (x_{i_1}\dots x_{i_d}x_i).\vtx{v}_0$, for $i\in q$.
	\end{itemize} Define 
\begin{equation}
\label{eq:define_colour}
c_\vtx{w}(\vtx{w},\vtx{w}^+) = q\mbox{ and }c_\vtx{w}(\vtx{w},(x_{i_1}\dots x_{i_d}).\vtx{w}_i) = i.
\end{equation}
Then $c_\vtx{w}$ satisfies Definition~\ref{defn:compatible_colour}.\ref{defn:compatible_colour1} and, setting $x = x_{i_1}\dots x_{i_d}$, Definition~\ref{defn:compatible_colour}.\ref{defn:compatible_colour3} is satisfied for $\vtx{u}_1 = \vtx{v}_0$ and $\vtx{u}_2 = \vtx{w}$. 

To extend the labelling to all of $\gr{E}\treeq$, observe that $x_0.\tree_{\vtx{v}_0}\subset \tree_{\vtx{v}_0}$ with translation by $x_0$ preserving the labelling in $\gr{E}\tree_{\vtx{v}_0}$, and that every vertex in $\treeq$ is in $\tree_{\vtx{v}_n}$ for some $n\in\mathbb{Z}$. Hence, if $\vtx{v}\in \tree_{\vtx{v}_n}$, then $x_0^{-n}.\vtx{v}\in \tree_{\vtx{v}_0}$ and it does not depend on the choice of $n$ to define 
\begin{equation}
\label{eq:define_colour2}
c_\vtx{v}(\vtx{v},\vtx{w}) = c_{x_0^{-n}.\vtx{v}}(x_0^{-n}.\vtx{v}, x_0^{-n}.\vtx{w}), \quad (\vtx{v},\vtx{w})\in \gr{E}_\vtx{v}.
\end{equation}

To see that Definition~\ref{defn:compatible_colour}.\ref{defn:compatible_colour3} is satisfied, consider $\vtx{u}_1,\vtx{u}_2\in \gr{V}\treeq$. Then there are $n_i\geq0$ such that $x_0^{n_i}.\vtx{u}_i\in \tree_{\vtx{v}_0}$ and $y_i\in P$ such that $x_0^{n_i}.\vtx{u}_i = y_i.\vtx{v}_0$ and $x_0$, $y_1$ and $y_2$ preserve the labelling. Set $x = x_0^{-n_2}y_2y_1^{-1}x_0^{n_1}$. Then $x.\vtx{u}_1 = \vtx{u}_2$ and $x|_{\tree_{\vtx{u}_1}}$ preserves the labelling $\tree_{\vtx{u}_1}\to \tree_{\vtx{u}_2}$. 
\end{proof}

The following illustrates the labelling defined in Proposition~\ref{prop:compatible_label}.
\begin{example}
\label{examp:p-adic}
Labellings of $\gr{E}\tree_{p+1}$ compatible with the action of $\mathbb{Q}_p \rtimes\mathbb{Q}_p^\times$ on $\tree_{p+1}$ given in Equation~\eqref{eq:padic_action} of Example~\ref{examp:padic} will be defined. 

Denote $0+p\mathbb{Z}_p\in \gr{V}\tree_{p+1}$ by $\vtx{v}_0$. 
Choose $a\in \mathbb{Z}_p^\times$ with $|a|_p=1$ and set $x_i = (ip,ap)$ in $\mathbb{Q}_p \rtimes\mathbb{Q}_p^\times$. The children of $\vtx{v}_0$ in $\tree_{\vtx{v}_0}$ are 
$$
\vtx{w}_i = x_i.(0+p\mathbb{Z}_p) = ip + p^2\mathbb{Z}_p,\quad i\in\{0,\dots,p-1\}.
$$ 

 Following Proposition~\ref{prop:compatible_label}, we begin by labelling the edges from the vertex, $\vtx{w} = x.\vtx{v}_0$, with $x = x_{i_1}\dots x_{i_d}$, on level $d$ of the rooted $\tree_{\vtx{v}_0}$. Noting that $x = (b,(ap)^d)$ where $b = i_1p+ i_2 ap^2 +\dots + i_d a^{d-1}p^d$, we have $\vtx{w} = b + p^{d+1}\mathbb{Z}$ and the children of $\vtx{w}$ on level $d+1$ are $x.\mathbf{w}_i = b + ia^dp^{d+1} + p^{d+2}\mathbb{Z}_p$. Hence, denoting the labelling given by this choice of $x_i$ in Proposition~\ref{prop:compatible_label} by $c^{(a)}_{\mathbf{w}}$,
\begin{align}
 c^{(a)}_{\mathbf{w}}(b+p^{d+1}\mathbb{Z}_p,b + ia^dp^{d+1} + p^{d+2}\mathbb{Z}_p) &= i \ \ \text{ or, equivalently, }\notag\\
c^{(a)}_{\mathbf{w}}(b+p^{d+1}\mathbb{Z}_p,b + jp^{d+1} + p^{d+2}\mathbb{Z}_p) &= ja^{-d} \pmod{p}.\label{eq:label_tree}
\end{align}
In case $a=1$ this labelling also comes from the identification of $\gr{V}\tree_{p+1}$ with $\bigcup_{n\in\mathbb{Z}} \mathbb{Q}_p/p^{n+1}\mathbb{Z}_p$.

The labelling extends to all of $\tree_{p+1}$ by Equation~\eqref{eq:define_colour2}. Given $\mathbf{v}\in \gr{V}_d\tree_{p+1}$ and $(\mathbf{v},\mathbf{w})\in \gr{E}_{\mathbf{v}}$, we have $\mathbf{v} = b+p^{d+1}\mathbb{Z}_p$  and $\mathbf{w} = b+jp^{d+1} + p^{d+2}\mathbb{Z}_p$. Then, choosing $n$ such that $x_0^{-n}.\mathbf{v} \in \tree_{\mathbf{v}_0}$, we have $x_0^{-n}.\mathbf{v} = b(ap)^{-n}+p^{d-n+1}\mathbb{Z}_p$ and $x_0^{-n}.\mathbf{w} = b(ap)^{-n}+ja^{-n}p^{d-n+1} + p^{d-n+2}\mathbb{Z}_p$. Hence, by Equations~\eqref{eq:define_colour2} and~\eqref{eq:label_tree},
$$
c^{(a)}_{\mathbf{v}}(\mathbf{v},\mathbf{w}) = c^{(a)}_{x_0^{-n}.\vtx{v}}(x_0^{-n}.\vtx{v}, x_0^{-n}.\vtx{w}) = ja^{-n}a^{n-d} \pmod{p} = ja^{-d} \pmod{p},
$$ 
and Equation~\eqref{eq:label_tree} extends to all of $\tree_{p+1}$. 
\end{example}

\begin{remark}
The $P$-labelling of $\gr{E}\treeq$ in Proposition~\ref{prop:compatible_label} involves a choice of elements $x_i\in P$, $i\in\{0,\dots,q-1\}$. Different choices may produce a different $P$-labellings. In fact, all $P$-labellings correspond to such choices. Given a $P$-labelling $\{c_{\vtx{v}}\}_{\vtx{v}\in \gr{V}\treeq}$, let $\vtx{v}_0\in\gr{V}\treeq$ and $\vtx{w}_0$, \dots, $\vtx{w}_{q-1}$ be its  children in $\tree_{\vtx{v}_0}$. For each $i\in\{0,\dots,q-1\}$ choose $x_i\in P$ such that $x_i.\vtx{v} = \vtx{w}_i$ and $x_i|_{\tree_{\vtx{v}}} :  \tree_{\vtx{v}}\to \tree_{\vtx{w}_i}$ preserves the labelling of $\tree_{\vtx{v}}$. Then using these $x_0$, \dots, $x_{q-1}$ in the proof of Proposition~\ref{prop:compatible_label} reproduces the given labelling of $\gr{E}\treeq$. 
\end{remark}

The definition of $\treeq$ in terms of semi-infinite strings in Definition~\ref{defn:treeq+1} leads to the following labelling of $\gr{E}\treeq$.
\begin{definition}
	\label{defn:stan_labelling}
	The \emph{standard labelling} of $\wtreeq$ is $\{c_{\vtx{v}}\}_{\vtx{v}\in\gr{V}(\wtreeq)}$ with 
	$$
	c_\vtx{v}(\vtx{v},\vtx{v}j) = j\mbox{ and } c_\vtx{v}(\vtx{v},\vtx{v}^+) = q,\quad (j\in\{0,\dots, q-1\}).
	$$ 
\end{definition}
The standard labelling is an $\Isom(\treeq)_{\omega}$-labelling of $\gr{E}\treeq$ with $\vtx{v}_n = \tvtx{v}_n$, but it is also true that every labelling satisfying Definition~\ref{defn:compatible_colour}.\ref{defn:compatible_colour1} and~\ref{defn:compatible_colour}.\ref{defn:compatible_colour2} is compatible with $\Isom(\treeq)_{\omega}$. A proper scale subgroup, $P$, of $\Isom(\treeq)_{\omega}$ naturally has fewer $P$-labellings. All $P$-labellings are conjugate to the standard one, as follows.  

\begin{proposition}
	\label{prop:xzero}
	Let $P\in\inftran$ and suppose that $\{c'_{\vtx{v}}\}_{\vtx{v}\in \gr{V}\treeq}$ is a $P$-labelling of $\gr{E}\treeq$. Then there is an automorphism, $x$, of $\treeq$ that sends $\{c'_{\vtx{v}}\}_{\vtx{v}\in \gr{V}\treeq}$ to the standard labelling, {\it i.e.\/} $c'_{\vtx{v}} = c_{x.\vtx{v}}\circ x$, and with $x.\vtx{v}_n = \tvtx{v}_n$. 
	
	The conjugate $P' = x P x^{-1}$ belongs to $\inftran$ and the standard labelling is compatible with $P'$.
\end{proposition}
\begin{proof}
	Since $\Isom(\treeq)$ is transitive on $\partial\treeq$, it may be supposed that $P$ stabilises the end $\omega$ specified in Definition~\ref{defn:treeq+1}. Then $x$ will be in $\Isom(\treeq)_\omega$ and we may use Proposition~\ref{prop:describe_autos} to define it.
	
	Let $\vtx{v}\in \gr{V}\treeq$. Then ${c}_{\vtx{v}}(\vtx{v},\vtx{v}^+) = {c'}_{\vtx{v}}(\vtx{v},\vtx{v}^+) =q$ and so the restriction of ${c'}_{\vtx{v}}\circ{c}_{\vtx{v}}^{-1}$ to $\{0,1\dots,q-1\}$ is a permutation. Denote this permutation by $\pi_{\vtx{v}}$. Then there is $N$ such that $\pi_{\vtx{v}_n}(0)=0$ for all $n\leq N$ because, by Definition~\ref{defn:compatible_colour}.\ref{defn:compatible_colour2}, the ray $\{\vtx{v}_{n}\}_{n\leq0}$ belongs to $\omega$ and $N$ and $l$ may be chosen such that $\vtx{v}_n = \tvtx{v}_{n+l}$ for all $n\leq N$. Defining $x$ as in Proposition~\ref{prop:describe_autos} we obtain, by~\eqref{eq:auto_form0} and the definitions of $c_{x.\vtx{v}}$ and $\pi_\vtx{v}$,
	$$
	c_{x.\vtx{v}}(x.\vtx{v},x.(\vtx{v}j)) = c_{x.\vtx{v}}(x.\vtx{v},x.\vtx{v}\pi_\vtx{v}(j)) = \pi_\vtx{v}(j) = \pi_\vtx{v}(c_\vtx{v}(\vtx{v},\vtx{v}j)) = c'_\vtx{v}(\vtx{v},\vtx{v}j)
	$$
	as claimed. Since $\pi_{\vtx{v}_{n}}(0)=0$ for all $n\leq N$, it follows from Equation~\eqref{eq:auto_form2} that $x.\vtx{v}_{n}=\vtx{v}_{n}=\tvtx{v}_{n+l}$ for all such $n$. Replacing $x$ by $\tilde{x}_0^{-l}x$, it may be supposed that $x.\vtx{v}_{n}=\tvtx{v}_{n}$ for $n\leq N$. For any $n$ for which this holds, we have $\vtx{v}_{n+1} = \vtx{v}_nk$ with $k\in\{0,1,\dots, q-1\}$ and $c'_{\vtx{v}_n}(\vtx{v}_n,\vtx{v}_nk) = 0$. Hence, by Equation~\eqref{eq:auto_form2}, $x.\vtx{v}_{n+1}= x.\vtx{v}_n\pi_{\vtx{v}_n}(k)$ with $\pi_{\vtx{v}_n}(k) = \pi_{\vtx{v}_n}(c_{\vtx{v}_n}(\vtx{v}_n,\vtx{v}_nk)) = c'_{\vtx{v}_n}(\vtx{v}_n,\vtx{v}_nk) = 0$ and it follows that $x.\vtx{v}_{n+1}=\tvtx{v}_n0 = \tvtx{v}_{n+1}$. An induction argument now shows that $x.\vtx{v}_n = \tvtx{v}_{n}$ for all $n\in\mathbb{Z}$.
	
	The conditions satisfied by scale groups are preserved under conjugation and so $P' = x P x^{-1}$ is a scale group. The standard labelling satisfies Definition~\ref{defn:compatible_colour}.\ref{defn:compatible_colour1}\&.\ref{defn:compatible_colour2}, and it remains only to check the third condition for $P'$. Let $\vtx{u}_1,\vtx{u}_2\in \gr{V}\treeq$. There is $g\in P$ such that $g.(x^{-1}.\vtx{u}_1)= x^{-1}.\vtx{u}_2$ and $g|_{\tree_{x^{-1}.\vtx{u}_1}}$ preserves the labelling $c'$ because $c'$ is compatible with $P$. Then $xgx^{-1}\in P'$, $(xgx^{-1}).\vtx{u}_1=\vtx{u}_2$ and $xgx^{-1}|_{\tree_{\vtx{u}_1}}$ preserves the standard labelling.
\end{proof}

The elements $x_i\in P$, $i\in \{0,\dots,q-1\}$, chosen in the proof of Proposition~\ref{prop:compatible_label} generate a subgroup of $P$ that is transitive on $\gr{V}\treeq$. The proof thus also establishes the following fact which will be useful later.
\begin{proposition}
	\label{prop:weak_criterion}
	Suppose that $P$ is a closed subgroup of $\Isom(\treeq)_\omega$ and that, for some $\vtx{v}\in \gr{V}\treeq$, there are $x_i\in P$, $i\in\{0,\dots,q-1\}$, such that $x_i.\vtx{v}$ are the children of $\vtx{\vtx{v}}$ in the rooted tree $\tree_\vtx{v}$. Then $P$ is a scale group. 
	\endproof
\end{proposition}

\subsection{$P$-labellings and self-replicating groups}
\label{sec:scaleands.r.}

The correspondence between a scale group, $P$, acting on $\treeq$ with a compatible $P$-labelling and a self-replicating group, $\Psr{P}$, acting on $\treeqq$ will now be exhibited. Recall the definition of the section $x|_{\vtx{v}}$ from paragraph~\ref{num:section_regular}.


Given a scale group, $P$, put $P|_\vtx{v} := \left\{ x|_{\vtx{v}}\in \Isom(\tree_{q,q}) \mid x\in P\right\}$. 
\begin{proposition}
\label{prop:s.s.&s.r.}
Let $P\in \inftran$ be compatible with the standard labelling of $\gr{E}\treeq$. Then $P|_\vtx{v}$ is independent of $\vtx{v}\in \gr{V}\treeq$, is equal to $\phi_{\vtx{v}}\circ P_\vtx{v}|_{\tree_{\vtx{v}}}\circ\phi_{\vtx{v}}^{-1}$, and is a closed self-replicating subgroup of $\Isom(\tree_{q,q})$. 
\end{proposition}
\begin{proof}
The first step in the argument is to show that $P|_\vtx{v} = \phi_{\vtx{v}}\circ P_\vtx{v}|_{\tree_{\vtx{v}}}\circ\phi_{\vtx{v}}^{-1}$. It is clear that $P|_\vtx{v} \supseteq \phi_{\vtx{v}}\circ P_\vtx{v}|_{\tree_{\vtx{v}}}\circ\phi_{\vtx{v}}^{-1}$. Consider $x|_\vtx{v}\in P|_\vtx{v}$. Since $P\in \inftran$, there is $z\in P$ such that $z.(x.\vtx{v}) = \vtx{v}$ and $z|_{\tree_{x.\vtx{v}}}$ preserves the labelling $\tree_{x.\vtx{v}}\to \tree_\vtx{v}$. Then $zx\in P_\vtx{v}$ and $x|_\vtx{v} =  zx|_\vtx{v}$. Hence $P|_\vtx{v} \subseteq \phi_{\vtx{v}}\circ P_\vtx{v}|_{\tree_{\vtx{v}}}\circ\phi_{\vtx{v}}^{-1}$. 

That $P|_\vtx{v} = \phi_{\vtx{v}}\circ P_\vtx{v}|_{\tree_{\vtx{v}}}\circ\phi_{\vtx{v}}^{-1}$ is a closed subgroup of $\Isom(\tree_{q,q})$ then follows because $P_\vtx{v}$ is a closed subgroup of $\Isom(\treeq)$. 

Suppose next that $\vtx{w}$ is another vertex in $\treeq$. Then, since $P\in \inftran$, there is $y\in P$ such that $y.\vtx{w} = \vtx{v}$ and $y|_{\tree_\vtx{w}} : \tree_\vtx{w}\to \tree_\vtx{v}$ preserves the labelling of $\tree_\vtx{w}$. Hence for each $x\in P_\vtx{v}$ we have that $y^{-1}xy\in P_\vtx{w}$ and $y^{-1}xy|_{\tree_\vtx{w}} =  x|_{\tree_\vtx{v}}$.
It follows that $P_\vtx{v}|_{\tree_{\vtx{v}}}\leq P_\vtx{w}|_{\tree_{\vtx{w}}}$ and, reversing roles, that $P_\vtx{v}|_{\tree_{\vtx{v}}}\geq P_\vtx{w}|_{\tree_{\vtx{w}}}$. Hence $P|_\vtx{v} = P|_\vtx{w}$ and does not depend on $\vtx{v}$. 

To show that $\phi_{\vtx{v}}\circ P_\vtx{v}|_{\tree_{\vtx{v}}}\circ\phi_{\vtx{v}}^{-1}$ is self-replicating, it must be shown that 
$$
(\phi_{\vtx{v}}\circ P_\vtx{v}|_{\tree_{\vtx{v}}}\circ\phi_{\vtx{v}}^{-1})|_{\vtx{s}} = \phi_{\vtx{v}}\circ P_\vtx{v}|_{\tree_{\vtx{v}}}\circ\phi_{\vtx{v}}^{-1}\mbox{ for every }{\vtx{s}}\in \tree_{q,q}.
$$
This holds because, recalling that $\vtx{v}$ is a semi-infinite string and $\vtx{s}$ is finite string, we have that $(\phi_{\vtx{v}}\circ P_\vtx{v}|_{\tree_{\vtx{v}}}\circ\phi_{\vtx{v}}^{-1})|_{\vtx{s}} = P|_{\vtx{v}\vtx{s}}$, with the concatenation $\vtx{v}\vtx{s}$ a semi-infinite string, and it has already been shown that $P|_{\vtx{v}\vtx{s}} = P|_\vtx{v}$. 
\end{proof}

Denoting the self-replicating group $P|_{\vtx{v}}$ acting on $\treeqq$ constructed in Proposition~\ref{prop:s.s.&s.r.} by $\Psr{P}$, we have established one direction of the correspondence $P\leftrightarrow \Psr{P}$. The other direction is established next.
\begin{proposition}
\label{prop:s.s.&s.r.2}
Let $\Psr{P}$ be a closed self-replicating subgroup of $\Isom(\tree_{q,q})$. Then there is $P\in \inftran$ compatible with the standard labelling of $\wtreeq$ and with $\Psr{P} = P|_{\vtx{v}}$ for every $\vtx{v}\in\gr{V}\treeq$. 
\end{proposition}
\begin{proof}
Let the section $|_\vtx{v}:\Isom(\treeq)\to\Isom(\treeqq)$ be as in~\eqref{eq:section} and define
\begin{equation}
\label{eq:define_P}
P = \left\{ x\in\Isom(\wtreeq) \mid x|_\vtx{v} \in \Psr{P}\mbox{ for all }\vtx{v}\in \gr{V}\wtreeq\right\}.
\end{equation}
Then $P$ is closed in $\Isom(\treeq)$ because the section is continuous and $\Psr{P}$ is closed in~$\Isom(\tree_{q,q})$, and is a subgroup because, for all $x,y\in\Isom(\treeq)$, $xy|_\vtx{v}\ =  x|_{y.\vtx{v}}y|_\vtx{v}$ and $\Psr{P}$ is a subgroup of $\Isom(\treeqq)$. 

To show that $\Psr{P} = P|_{{\vtx{v}}}$ it suffices to show that $\Psr{P} \leq P|_{{\vtx{v}}}$, because the reverse inclusion holds by definition, and, for this, it suffices to show that, given $\hat{g}\in \Psr{P}$, there is $g\in P$ such that $g|_{\vtx{v}_0} = \hat{g}$. Define, for $\vtx{v}\in \gr{V}\tree_{\vtx{v}_0}$, 
$$
g.\vtx{v} = \phi_{\vtx{v}_0}^{-1}(\hat{g}.\phi^{\phantom{-1}}_{\vtx{v}_0}(\vtx{v}))
$$ 
with $\phi_{\vtx{v}_0}$ as in paragraph~\ref{num:section_regular}. The fact that $\Psr{P}$ is self-replicating may now be used to extend $g$ inductively to $\tree_{\vtx{v}_{-n}}$ for each $n\geq0$ so that $g.\vtx{v}_0 = \vtx{v}_0$ and $\phi_{\vtx{v}_{-n}}\circ g\circ \phi_{\vtx{v}_{-n}}^{-1}\in \Psr{P}$. For the inductive step, note that the map $\phi_{\vtx{v}_{-n-1}}$ is an isomorphism $\tree_{\vtx{v}_{-n-1}}\to \treeqq$ which carries $\tree_{\vtx{v}_{-n}}$ to the subtree,  $\tree_\vtx{0}$, of $\treeqq$ spanned by $\left\{\vtx{v} = v_1v_2\dots v_r \in q^*\mid v_1=0\right\}$, see paragraph~\ref{num:rooted_section}. Since $\Psr{P}$ is self-replicating, there is $\hat{g}^{(n+1)}\in \Psr{P}$ such that $\hat{g}^{(n+1)}.\vtx{0} = \vtx{0}$ and $\hat{g}^{(n+1)}|_{\tree_{\vtx{0}}} = \phi_{\vtx{v}_{-n}}\circ g\circ \phi_{\vtx{v}_{-n}}^{-1}$. Then $\phi_{\vtx{v}_{-n-1}}\circ \hat{g}^{(n+1)}\circ \phi_{\vtx{v}_{-n-1}}$ is an extension of $g$ to $\tree_{\vtx{v}_{-n-1} }$ which stabilises $\vtx{v}_0$ and is such that $\phi_{\vtx{v}_{-n-1}}\circ g\circ \phi_{\vtx{v}_{-n-1}}^{-1} = \hat{g}^{(n+1)}\in \Psr{P}$. Since $\treeq = \bigcup_{n\geq0} \tree_{\mathbf{v}_{-n}}$, this induction extends $g$ to all of $\treeq$. Every $\vtx{v}$ belongs to $\tree_{\mathbf{v}_{-n}}$ for large enough $n$ and, for such $n$, $g|_\vtx{v}$ is a section of $\Psr{P}$ because $\phi_{\vtx{v}_{-n}}\circ g\circ \phi_{\vtx{v}_{-n}}^{-1} \in \Psr{P}$. Hence $g$ belongs to $P$ as required.

 It remains to show that $P$ is a scale group. The horosphere $\gr{V}_{0}$ is the increasing union of its intersections with $\tree_{\vtx{v}_{-n}}$ as $n\to\infty$. The intersection $\gr{V}_0\cap\tree_{\vtx{v}_{-n}}$ is precisely the $n^{\rm th}$ level of $\tree_{\vtx{v}_{-n}}$and so, since $\Psr{P}$ is transitive on each level of $\treeqq$, $\gr{V}_0$ is contained in a single $P$-orbit in $\treeq$. Next, notice that $\tilde{x}_0$ belongs to $P$ because $\tilde{x}_0|_\vtx{v} = \id$ for every $\vtx{v}\in \gr{V}(\treeq)$. Since $\tilde{x}_0^n.\gr{V}_0 = \gr{V}_n$, it follows that $P$ is transitive on $\gr{V}\treeq$ and hence is a scale group. 
\end{proof}
The maps $P\to\Psr{P}$ and $\Psr{P}\to P$ defined in Propositions~\ref{prop:s.s.&s.r.} and~\ref{prop:s.s.&s.r.2} are inverse to one another and there is thus a bijective correspondence $P\leftrightarrow \Psr{P}$ between scale groups with a given $P$-labelling and self-replicating groups. Illustrating this correspondence: taking $P = \Isom(\treeq)_\omega$, we have $\Psr{P} = \Isom(\treeqq)$; taking $P = H\rtimes\langle x_0\rangle$ as in Example~\ref{examp:scale_group}, $\Psr{P}$ is the image of the coordinatewise action of $S^{\mathbb{N}}$ on finite strings in $\{0,\dots,q-1\}$; and taking $\mathbb{Q}_p\rtimes\langle p\rangle$ as in Example~\ref{examp:p-adic}, $\Psr{P}$ is the closure of $\mathbb{Z}\curvearrowright\tree_{p+1}$ under the odometer action,~\cite[\S1.3.4]{Nekrash}, which is isomorphic to $(\mathbb{Z}_p,+)$. 

\subsection{Residue actions of a scale group}
\label{sec:residual}

Self-replicating groups are residually finite because they act on rooted trees. Each such group,~$\Psr{P}$, preserves the levels of the tree and hence there is, for each~$d\geq0$, a homomorphism~$\pi_d : \Psr{P}\to \Isom(\treeqq^{(d)})$, where $\treeqq^{(d)}$ is the  finite subtree of $\treeqq$ with depth~$d$. 

The correspondence $P\leftrightarrow \Psr{P}$ associates the finite groups $\pi_d(\Psr{P})$ with the scale group~$P$ and a given $P$-labelling of $\gr{E}\treeq$. These groups are not quotients of~$P$ but bear a similar relationship to $P$ as the residual field does to $\mathbb{Q}_p$ and will be called residue groups of~$P$. Residue groups are finite invariants of scale groups and will be used in Section~\ref{sec:tree_rep} to distinguish between them. 
\begin{definition}
\label{defn:residualaction}
The \emph{residue group} of $P\in\inftran$ is the permutation group on $\{0,\dots,q-1\}$ induced by the action of $\Psr{P}$ on level~$1$ of $\treeqq$, up to conjugation in $\Sym(q)$. The residue group of $P$ is denoted by $[P]$. 

More generally, the \emph{level~$d$ residue group} of $P\in\inftran$ is a subgroup of $\Isom(\treeqq^{(d)})$ induced by the action of $\Psr{P}$ on $\treeqq^{(d)}$, up to conjugation in $\Isom(\treeqq^{(d)})$. The level~$d$ residue group is denoted by $[P]^{(d)}$. 
\end{definition}

Self-replicating groups can be distinguished by seeing that their level $d$ residue groups differ for one~$d$. For this, it suffices to consider their actions on $\{0,\dots, q-1\}^d$, the set of leaves of $\treeqq^{(d)}$, or to distinguish their level $d$ residues as abstract groups. The next proposition facilitates that without needing to consider the action on $\treeqq^{(d)}$.
\begin{proposition}
\label{prop:dresidisquotient}
\begin{enumerate}
\item \label{prop:dresidisquotient1}
Let $\Psr{P}$ be a self-replicating group and let $\vtx{w}\in\gr{V}\treeqq$ be on level~$d$. Then the action of $\Psr{P}$ on $\{0,\dots, q-1\}^d$ is equivalent to the translation action of $\Psr{P}$ on the coset space $\Psr{P}/\Psr{P}_{\vtx{w}}$
\item \label{prop:dresidisquotient2}
Let $P\in\inftran$ and suppose that $x\in P$ has $s(x^{-1}) = q$. Then there is $\vtx{v}\in\gr{V}\treeq$ such that the action of $P_{\vtx{v}}$ on $P_{\vtx{v}}/(x^dP_{\vtx{v}}x^{-d})$ determines the same permutation group as the action of $[P]^{(d)}$ on $\{0,\dots, q-1\}^d$.
\end{enumerate}
\end{proposition}
\begin{proof} 
 \ref{prop:dresidisquotient1}. Self-replicating groups are transitive on each level of $\treeqq$ by definition. Hence the map $g\Psr{P}_{\vtx{w}}\mapsto g.\vtx{w}$ is a bijection $\Psr{P}/\Psr{P}_{\vtx{w}}\to \{0,\dots, q-1\}^d$ and is $\Psr{P}$-equivariant.
 
\ref{prop:dresidisquotient2}. Since $s(x^{-1}) = q$, $x$ is a translation in the direction away from $\omega$. Choose $\vtx{v}$ on the axis of translation by~$x$. Then $x.\vtx{v}$ is a child of $\vtx{v}$ in $\tree_{\vtx{v}}$ and we may choose $x_1,\ \dots,\ x_{q-1}\in P$ that send $\vtx{v}$ to each of the other children of $\vtx{v}$ and define a $P$-labelling of $\treeq$ that is preserved by~$x$ as in Proposition~\ref{prop:s.s.&s.r.}. Then $x^dP_{\vtx{v}}x^{-d} = P_{x^d.\vtx{v}}$ and, moreover, $P_{x^d.\vtx{v}}$ contains the kernel of the restriction map $g\mapsto g|_{\tree_\vtx{v}}: P_{\vtx{v}} \to \Psr{P}$. This map is a surjective homomorphism, by Proposition~\ref{prop:s.s.&s.r.}, and hence induces a $P_{\vtx{v}}$-equivariant bijection $P_{\vtx{v}}/P_{x^d.\vtx{v}} \to \Psr{P}/\Psr{P}_{\vtx{w}}$ with $\vtx{w} = \phi_{\vtx{v}}(x^d.\vtx{v})$.
The claim follows from part~\ref{prop:dresidisquotient1}. 
\end{proof}

  \begin{remark}
\label{rem:residue}
The scale groups seen in Examples~\ref{examp:scale_group2} and~\ref{examp:scale_group} have residue group equal to~$S$. Indeed, it may be seen that $P^{(S)}$ of Example~\ref{examp:scale_group2} is a maximal scale group with residue~$S$, and that $H\rtimes \langle x_0\rangle$ of Example~\ref{examp:scale_group} is minimal with this residue. 
The groups $\mathbb{Q}_p\rtimes\langle ap\rangle$ of Example~\ref{examp:p-adic} are not all isomorphic but all have residue $(\mathbb{F}_p,+)$, while $\mathbb{Q}_p\rtimes\mathbb{Q}^\times_p$ has residue $\mathbb{F}_p\rtimes\mathbb{F}^\times_p$. These groups are minimal for their respective residues.
\end{remark}

\section{The tree representation theorem}
\label{sec:tree_rep}

The tree representation theorem, \cite[Theorem~4.1]{BW}, associates a scale group to a subgroup tidy for an element of a t.d.l.c.~group. Concrete isometry groups of trees thus emerge from the abstract theory of t.d.l.c.~groups. This section describes the representation in terms of the tree labellings given above, and examines the dependence of the scale group on the choice of tidy subgroup. 

A few definitions are needed for the statement the theorem. For $G$ a t.d.l.c.~group, $x\in G$ and a compact open subgroup $V\leq G$, define subgroups 
$$
V_{\pm} := \bigcap_{n\geq0} x^{\pm n} Vx^{\mp n}\mbox{ and }V_{--} := \bigcup_{n\in\mathbb{Z}} x^n V_-x^{-n}.
$$ 
Then $V$ is \emph{tidy for~$x$} if $V = V_+V_-$ and $V_{--}$ is closed. Subgroups tidy for~$x$ exist, and the \emph{scale of~$x$}, which is a structural invariant for $G$ and does not depend on~$V$, is the positive integer $s(x) = [xV_+x^{-1} : V_+]$, see~\cite{GWil}. 

Tidy subgroups and the scale have been used to answer numerous questions about general t.d.l.c.~groups. The next theorem associates more detailed information with~$x$ than the scale value. 
\begin{theorem}[Tree Representation]
\label{thm:tree_rep}
Let $G$ be a t.d.l.c.~group, $x\in G$ and $V\leq G$ be a compact open subgroup tidy for~$x$. Suppose that $s(x^{-1}) = q$ is greater than $1$. Then there is a continuous homomorphism 
$$
\rho : V_{--}\rtimes \langle x\rangle \to \Isom(\treeq)_\omega
$$ 
with $\omega\in\partial\tree_{q+1}$. Moreover: $V$ stabilises a vertex in $\gr{V}\tree_{q+1}$; $\ker\rho$ is the largest compact normal subgroup of $V_{--}\rtimes \langle x\rangle$; and $\rho(V_{--}\rtimes \langle x\rangle)$ is a closed subgroup of $\Isom(\treeq)_\omega$ and is transitive on $\gr{V}\tree_{q+1}$.\endproof
\end{theorem}

Thus, a t.d.l.c.~group $G$ has a subquotient isomorphic to a scale group once it has an element, $x$, with $s(x^{-1})>1$. The homomorphism $\rho$ sends $x$ to a translation of $\treeq$ that has $\omega$ as a repelling end, and sends $V_{--}$ to a group of elliptic elements in $\Isom(\treeq)_\omega$. (If the definition of scale groups were expanded to include the group of translations of a line, then these claims would hold for $s(x^{-1})=1$ too but would not convey any information.)

Conversely, let $P$ be a scale group acting on $\treeq$ and choose $x\in P$ such that $x.\vtx{v}_0 = \vtx{v}_{1}$. Then $V := P_{\vtx{v}_0}$ is tidy for $x$ and $s(x^{-1}) = q$. Further, $V_{--}$ is equal to  $P^{(e)}$ and $P = V_{--}\rtimes \langle x\rangle$. 

The following characterises scale groups up to group isomorphism. It is similar in spirit to \cite[Theorems 4.1 \& 7.1]{CaCoMoTe}.
\begin{corollary}
\label{cor:tree_rep}
The locally compact group $G$ has a faithful representation as a scale group if and only if it satisfies:
\begin{enumerate}
\item \label{cor:tree_rep1}
$G$ has no non-trivial compact normal subgroup;
\item \label{cor:tree_rep2}
there is an open normal subgroup $H\triangleleft G$ such that $G/H \cong\mathbb{Z}$; and
\item \label{cor:tree_rep3}
there are a compact open subgroup $V\leq G$ and an element $x\in G$ such that $xVx^{-1} \leq V$ and $\bigcup_{n\in\mathbb{Z}} x^nVx^{-n} = H$.
\end{enumerate}
Any $G$ satisfying \ref{cor:tree_rep1}.--\ref{cor:tree_rep3}. is isomorphic to a scale subgroup of $\Isom(\treeq)_{\omega}$ with $q = \min \left\{s(y) \mid y\in G \mbox{ with }s(y)\ne1\right\}$. \endproof
\end{corollary}
The representation of $G$ as a scale group need not be unique. Different choices of the compact open subgroup $V$ in the tree representation theorem may produce different actions of $G$ on $\treeq$, as is seen in \S\ref{sec:Pnotunique}.  

\subsection{Construction of the tree and compatible labellings}
\label{sec:tree_construction}

Theorem~\ref{thm:tree_rep} is proved by constructing the tree $\treeq$ from $x$ and $V_{--}$. One approach is to observe that the group $V_{--}\rtimes \langle x\rangle$ is isomorphic to the HNN-extension $V_-*_\alpha$ with $\alpha: V_-\to V_-$ being the endomorphism $\alpha(v) = xvx^{-1}$ and that $\treeq$ is the Bass-Serre tree for this HNN-extension. Since the index $[V_-:xV_-x^{-1}] =s(x^{-1}) = q$ is finite, this tree is locally finite and has valency equal to $q+1$. An equivalent approach that constructs the tree directly from  $V_{--}$ (which does not have to be constructed in the course of the argument because it is given to us as a subgroup of $G$) is followed in~\cite{BW} and here. This approach also constructs the tree in a form that matches the concrete description in Definition~\ref{defn:treeq+1} and labels the edges as in Proposition~\ref{prop:compatible_label}.

	For convenience, the conjugation automorphism $y\mapsto xyx^{-1}$ will be denoted by $\alpha$ and the scale $s(x^{-1})$ by~$q$. Denote $\gr{W}_n = V_{--}/\alpha^{n+1}(V_-)$, put 
$$
\gr{W} = \bigsqcup_{n\in\mathbb{Z}} \gr{W}_n\mbox{ and }\gr{E} = \left\{ \{g\alpha^n(V_-),g\alpha^{n+1}(V_-)\}\in \gr{W}^2 \mid n\in\mathbb{Z},\ g\in V_{--}\right\};
$$
and define $\tree^{(\alpha)}$ to be the graph $(\gr{W},\gr{E})$. The map $g\alpha^{n+1}(V_-)\mapsto g\alpha^{n}(V_-)$ on $\gr{W}(\tree^{(\alpha)})$ is $q$-to-$1$. Hence $g\alpha^{n+1}(V_-)\in\gr{W}_n$ has $q+1$ neighbours, namely: $g\alpha^{n}(V_-)\in\gr{W}_{n-1}$; and~$g\alpha^{n+1}(g_i)\alpha^{n+2}(V_-)\in\gr{W}_{n+1}$ where $g_i$, ${0\leq i\leq q-1}$, are coset representatives for $V_-/\alpha(V_-)$. Hence $\tree^{(\alpha)}$ is isomorphic to $\treeq$. 

Choosing a fixed set  $\{g_i\}_{0\leq i\leq q-1}$ of coset representatives for $V_-/\alpha(V_-)$, with $g_0 = \id$, specifies an isomorphism $\phi:\tree^{(\alpha)}\to \wtreeq$ as follows. The set $\{g\alpha^{m+1}(V_-)\}_{m\leq n}$ is a ray in $\tree^{(\alpha)}$ for each $g\in V_{--}$, and there is $m$ such that 
$$
g\alpha^{m'+1}(V_-) = \alpha^{m'+1}(V_-) = g_0\alpha^{m'+1}(V_-)\text{ for all }m'\leq m
$$ 
because $V_{--}=\bigcup_{m\in\mathbb{Z}}\alpha^m(V_-)$ and  $\alpha^{m-1}(V_-)>\alpha^m(V_-)$. Furthermore, it follows by induction on $n-m$ that there are unique $i_k\in \{0,\dots, q-1\}$ for all $m\leq k\leq n$ such that
\begin{equation}
\label{eq:productforg}
g\alpha^{n+1}(V_-) = \alpha^{m+1}(g_{i_{m+1}})\dots\alpha^{n}(g_{i_n})\alpha^{n+1}(V_-).
\end{equation}
Define 
$$
\phi(g\alpha^{n+1}(V_-))= \{i_m\}_{m\leq n}\in \{0,\dots,q-1\}^{(-\infty,n]}.
$$ 
Then $\phi$ maps $\gr{W}(\tree^{(\alpha)})$ to $\gr{V}(\wtreeq)$ and is a graph isomorphism because 
$$
\phi(g\alpha^{n+1}(V_-))^+ = \phi(g\alpha^n(V_-)).
$$ 
In particular, the rays $\{g\alpha^{m+1}(V_-)\}_{m\leq n}$ all belong to the same end of $\tree^{(\alpha)}$.
%
%

Next we see that $V_{--}\rtimes \langle x\rangle$ acts on $\tree^{(\alpha)}$ as a scale group. The group $V_{--}$ acts transitively on each $\gr{W}_n$ by translation because $\gr{W}_n$ is the quotient of $V_{--}$ by $\alpha^{n+1}(V_-)$.
Define an action of $x$  on $\gr{W}$ by
$$ 
\label{eq:V--action2}
x.(g\alpha^{n+1}(V_-)) = \alpha(g\alpha^{n+1}(V_-)) = \alpha(g)\alpha^{n+2}(V_-)\quad (g\in V_{--},\ n\in\mathbb{Z}).
$$
The actions of $V_{--}$ and $x$ extend to the semi-direct product and so the fact that $x.\gr{W}_n = \gr{W}_{n+1}$ implies that $V_{--}\rtimes \langle x\rangle$ is transitive on $\gr{W}$. Since
\begin{align*}
x.\{g\alpha^n(V_0-),g\alpha^{n+1}(V_-)\} &= \{\alpha(g)\alpha^{n+1}(V_0-),\alpha(g)\alpha^{n+2}(V_-)\}\\
\intertext{ and, for every $h\in V_{--}$, }
h.\{g\alpha^n(V_0-),g\alpha^{n+1}(V_-)\} &= \{(hg)\alpha^n(V_0-),(hg)\alpha^{n+1}(V_-)\}.
\end{align*}
$V_{--}\rtimes\langle x\rangle$ acts by tree isometries. Let $\rho : V_{--}\rtimes\langle x\rangle \to \Isom(\tree^{(\alpha)})$ be the induced homomorphism. To see that $\rho(V_{--}\rtimes\langle x\rangle)$ is closed, it suffices to show that $\rho(V_{--})$ is closed. That is so because, if $\rho(g_n)\to w\in \Isom(\tree^{(\alpha)})$, with $g_n\in V_{--}$, then it may be supposed that $\rho(g_n).V_- = w.V_-$ for all $n$ whence, since $V_-$ is compact, $\{g_n\}_{n\geq0}$ has an accumulation point $g\in V_{--}$ and $w=\rho(g)$. This completes the proof that $\rho(V_{--}\rtimes \langle x\rangle)$ is a scale group.

The standard labelling of $\tree_{q+1}$ pulled back to $\tree^{(\alpha)}$ via $\phi$ gives 
\begin{align*}
c_{g\alpha^{n+1}(V_-)}(g\alpha^{n+1}(V_-),g\alpha^{n}(V_-)) &=q  \quad \text{ and}\\
c_{g\alpha^{n+1}(V_-)}(g\alpha^{n+1}(V_-),g\alpha^{n+2}(V_-)) &= i_{n+1} \text{ for }i\in\{0,\dots,q-1\},
\end{align*} 
with $i_{n+1}$ as in Equation~\eqref{eq:productforg}. This labelling is compatible with $V_{--}\rtimes\langle x\rangle$ because: $\{\mathbf{v}_n=\alpha^{n+1}(V_-)\}_{n\in\mathbb{Z}}$ satisfies $\mathbf{v}_n = \mathbf{v}_{n+1}^+$ and $c_{\mathbf{v}_n}(\mathbf{v}_n,\mathbf{v}_{n+1})=0$ for all $n\in\mathbb{Z}$; and, given $g\alpha^{k+1}(V_-)$ and $h\alpha^{n+1}(V_-)$ in $\gr{W}$, then $z = h\alpha^{n-k}(g^{-1})x^{n-k}$ satisfies $z.g\alpha^{k+1}(V_-) = h\alpha^{n+1}(V_-)$ and preserves the labelling if the coset representatives $h$ and $g$ are chosen to be the products on the right of Equation~\eqref{eq:productforg}.

\subsection{Uniqueness of scale group representations}
\label{sec:sc.gp.rep._ uniqueness}

The scale-group representation of $V_{--}\rtimes\langle x\rangle$ of Theorem~\ref{thm:tree_rep} need not be unique. Different choices of subgroup tidy for $x$ may yield different actions on $\treeq$ even if the scale groups themselves are isomorphic. Given a representation, $\rho$, corresponding to the tidy subgroup $V_-$, another is obtained by composing $\rho$ with conjugation by an element, $y$ say, of $V_{--}\rtimes\langle x\rangle$. This conjugate representation corresponds to the subgroup $yV_-y^{-1}$ that is tidy for $yxy^{-1}$ and, if $V_-$ is the stabiliser of $\vtx{v}\in\treeq$, then $yV_-y^{-1}$ is the stabiliser of $y.\vtx{v}$. We consider such conjugate representations to be equivalent. Examples of non-equivalent representations are seen in \S\ref{sec:Pnotunique}. 

A criterion for uniqueness of the scale-group representation up to conjugacy is established in this section. To avoid unnecessary annotations, $V_-$ will be denoted by $V$ and $V_{--}$ by $H$. Then the image of $H\rtimes\langle x\rangle$ under $\rho$ will be the scale group $P = P^{(e)}\rtimes\langle x_0\rangle$ with $P^{(e)} = \rho(H)$ and $x_0=\rho(x)$. 

Several preliminary results are needed.
\begin{lemma}
\label{lem:con_trans}
Let $P = P^{(e)}\rtimes\langle x_0\rangle$ be a scale group acting on $\treeq$ and let $\overline{\con(x_0)}$ denote the closure of the contraction subgroup for $x_0$. Then $\overline{\con(x_0)}$ is transitive on $\partial\treeq\setminus\{\omega\}$. 
\end{lemma}
\begin{proof} By \cite[Proposition~3.16]{BW}, $H = \con(x_0)V_0$ with $V_0 = \bigcap_{n\in\mathbb{Z}} x_0^nVx_0^{-n}$. Hence $V = (V\cap \con(x_0))(V\cap x_0Vx_0^{-1})$ and there are $g_0, \dots, g_{q-1}$ in $V\cap \con(x_0)$ which are coset representatives for $V/(V\cap x_0Vx_0^{-1})$. Then $\rho(\overline{\con(x_0)}\rtimes \langle x\rangle)$ is a scale group, by Proposition~\ref{prop:weak_criterion}, and $\rho(\overline{\con(x_0)})$  is transitive on $\partial\treeq\setminus\{\omega\}$, by Proposition~\ref{prop:transitive_on_ends}. 
\end{proof}

\begin{lemma}
\label{lem:conjugate_end}
Let $P= P^{(e)}\rtimes\langle x_0\rangle$ be a scale group and let $hx_0\in P$. Then there is $g\in\overline{\con_P(x_0)}$ such that $g(hx_0)g^{-1}$ and $x_0$ have the same axis of translation. Hence $h'x_0$ with $h' = gh(x_0g^{-1}x_0^{-1})$ have the same axis of translation.
\end{lemma}
\begin{proof}
The action of $hx_0$ in $\treeq$ is hyperbolic with repelling end $\omega$. Let $\omega_0,\,\omega_1$ be the attracting ends of $x_0$ and $hx_0$ respectively. Then there is $g$ in $\overline{\con_P(x_0)}$ such that $g.\omega_1 = \omega_0$ by Lemma~\ref{lem:con_trans}. Hence $g(hx_0)g^{-1}.\omega_0 = \omega_0$ and $x_0$ and $g(hx_0)g^{-1}$ translate the same axis. 
\end{proof}

 \begin{lemma}
 \label{lem:decreasing}
Let $G$ be a group and $y\in G$. Suppose that $V,W\leq G$ satisfy that $yWy^{-1}<W$ and that there are $m,n\in\mathbb{Z}$ such that
\begin{equation}
\label{eq:decreasing}
y^mVy^{-m} \leq W \leq y^nVy^{-n}.
\end{equation}
Then there is $l\in\mathbb{Z}$ such that $yVy^{-1}\leq y^lWy^{-l}$ and $V\not\leq y^lWy^{-l}$. 
 \end{lemma}
 \begin{proof}
A rearrangement~\eqref{eq:decreasing} yields $y^{1-n}Wy^{n-1} \leq yVy^{-1}\leq y^{1-m}Wy^{m-1}$. Since $yWy^{-1}<W$, the first inclusion implies that $y^{k}Wy^{-k} < yVy^{-1}$ for every $k> 1-n$ and hence that there is a largest integer, $l$, such that $yVy^{-1}\leq y^lWy^{-l}$. Then, since~$l$ is the largest, $V\not\leq y^lWy^{-l}$. 
 \end{proof}

\begin{proposition}
\label{prop:sc.gp.rep._uniqueness} 
The scale group representation $\rho : H\rtimes \langle x\rangle\to \Isom(\treeq)$ is unique up to conjugacy if and only if the action $V\curvearrowright V/xVx^{-1}$ is primitive. 
\end{proposition}
\begin{proof}
We will show the contrapositive, that $H\rtimes \langle x\rangle$ has scale group representations that are not conjugate if and only if the action $V\curvearrowright V/xVx^{-1}$ is not primitive. Since every scale-group representation has kernel equal to the largest compact normal subgroup of $H\rtimes \langle x\rangle$, which is contained in $H$, we may divide out by this subgroup and assume that all such representations are faithful. 

Suppose that $V\curvearrowright V/xVx^{-1}$ is not primitive. Let $W$ be the stabiliser of the block of imprimitivity containing the coset $xVx^{-1}$. Then $W$ is compact and open in $H\rtimes \langle x\rangle$ and is tidy for $x$ because $xWx^{-1} < xVx^{-1} < W < V$. The inclusion $xVx^{-1} <W$ implies that $\bigcup_{n\in\mathbb{Z}} x^nWx^{-n} = H$ and so, by Theorem~\ref{thm:tree_rep}, there is a scale group representation $\rho_W : H\rtimes \langle x\rangle \to \Isom(\treeq)$ with $\rho_W(W)$ the stabiliser of a vertex in $\treeq$. The index $[V : W]$ is strictly less than $q$, whereas conjugates of $V$ contained in $V$ have index a power of~$q$. Therefore $W$ is not a conjugate of $V$ and $\rho_W$ is not conjugate to~$\rho$. 

Suppose, to the contrary, that there is $\rho_1$, a scale-group representation of $H\rtimes\langle x\rangle$, not conjugate to $\rho$. Then there is a compact open subgroup $W<H\rtimes\langle x\rangle$ such that $\rho_1(W)$ is the stabiliser of a vertex in $\treeq$ and, by Proposition~\ref{prop:scale_gp_props} and since $\rho_1$ is faithful, there is $hx\in H\rtimes\langle x\rangle$ such that $(hx)W(hx)^{-1}<W$. Conjugating by an element of $\overline{\con(x)}$, it may be assumed, by Lemma~\ref{lem:conjugate_end}, that $\rho(hx)$ and $\rho(x) = x_0$ have the same axis. Then $\rho(h)$ fixes all vertices on the axis of $x_0$, which implies that $\rho(h)\rho(xVx^{-1})\rho(h)^{-1} = \rho(xVx^{-1})$ and, since $\rho$ is faithful, that $(hx)V(hx)^{-1}=xVx^{-1}$. Hence $(hx)V(hx)^{-1}<V$ and we have, by Proposition~\ref{prop:scale_gp_props} again, that $\bigcup_{n\in\mathbb{Z}} (hx)^nV(hx)^{-n} = H = \bigcup_{n\in\mathbb{Z}} (hx)^nW(hx)^{-n}$. Then, since $V$ and $W$ are compact; there are $m,n\in\mathbb{Z}$ such that
$$
V\leq (hx)^m W(hx)^{-m}\mbox{ and }W\leq (hx)^nV(hx)^{-n}.
$$
By Lemma~\ref{lem:decreasing}, there is $l\in\mathbb{Z}$ such that $(hx)V(hx)^{-1}\leq (hx)^lW(hx)^{-l}$ and $V\not\leq (hx)^lW(hx)^{-l}$. However, it has already been seen that $(hx)V(hx)^{-1} = xVx^{-1}$ and so, since $W$ is not a conjugate of $V$, we have the strict inclusions
$$
xVx^{-1} < (hx)^lW(hx)^{-l}\cap V < V.
$$
Therefore the action $V\curvearrowright V/xVx^{-1}$ is not primitive. 
\end{proof}

Proposition~\ref{prop:sc.gp.rep._uniqueness} opens the way for the classification of finite primitive permutation groups given by the O'Nan-Scott Theorem, see~\cite{LiPrSa}, to have a role in the investigation of scale groups.  

\subsection{Examples of tree representation}
\label{sec:Pnotunique}

This subsection illustrates the tree representation theorem by calculating it for several groups and tidy subgroups. It is seen that the same group may have many non-equivalent scale group representations. 

\subsubsection{A general construction}
\label{sec:product_construction}
Here is a general construction on which the examples to follow are based. 

Let $F$ be a finite group and $A$ be a proper subgroup of $F$. Define 
$$
H = \left\{ h\in F^{\mathbb{Z}} \mid \exists n\in\mathbb{Z} \mbox{ such that }h_m \in A \mbox{ for all }m\leq n\right\}. 
$$
Equip $H$ with the coordinatewise product and the topology such that the profinite subgroups
$$
V_n := \left\{ h\in H \mid h_m \in A \mbox{ for all }m\leq n\right\} \cong A^{(-\infty,n]} \times F^{(n,\infty)} 
$$
are open. Define the shift automorphism $\alpha\in\Aut(H)$ by 
$$
\alpha(h)_n = h_{n-1},\quad (n\in\mathbb{Z}).
$$
Then $G(F,A) = H\rtimes \langle \alpha\rangle$ is a t.d.l.c.~group. In the case when $A$ is trivial, $G(F,A)$ will be abbreviated as $G(F)$.

Theorem~\ref{thm:tree_rep} will now be interpreted with $G=G(F,A)$, $x=\alpha$ and $V = V_0$. Then $V$ is tidy for~$\alpha$, because $\alpha^n(V) = V_n$ for all $n\in\mathbb{Z}$ and $V_+ = A^{\mathbb{Z}}$, $V_- = V$ and $V_{--} = H$, and $q = s(\alpha^{-1}) = [V_{-1} : V_0] = |F/A|$.  

The proof of the theorem identifies $\treeq$ with $\tree^{(\alpha)}$ having the vertex set
$$
\gr{V}\tree^{(\alpha)} = \gr{W} = \bigsqcup_{n\in\mathbb{Z}} H/V_n.
$$ 
Let $f_i$, $i\in\{0,\dots,q-1\}$, be coset representatives for $F/A$ with $f_0=\id_F$ and, abusing notation, denote $h\in H$ with $h_0=f_i$ and $h_m=\id_F$ for $m\ne 0$ by $f_i$ as well. Then, for $g = (g_m)\in H$ with $g_m\in A$ for $m<N$, we have 
$$
gV_{n+1} = \alpha^N(f_{i_N}) \dots \alpha^n(f_{i_n})V_{n+1},
$$ 
which may be compared with Equation~\eqref{eq:productforg}. Since the product in $H$ is coordinate-wise, the action of $H$ on $\gr{V}\tree^{(\alpha)}$ is
\begin{align*}
h.(gV_{n+1}) &= \alpha^M(h_{i_N}f_{i_N}) \dots \alpha^n(h_{i_n}f_{i_n})V_{n+1}\\
&= \alpha^M(f_{i'_N}) \dots \alpha^n(f_{i'_n})V_{n+1}
\end{align*}
with $M$ such that $h_m,g_m\in A$ for all $m<M$, and $f_{i'_m}$ the coset representative such that $h_mf_{i_m}A = f_{i'_m}A$. Hence the kernel of $\rho : G(F,A)\to \Isom(\tree^{(\alpha)})$ is $C^{\mathbb{Z}}$ where $C = \bigcap_{f\in F} fAf^{-1}$. 

The isomorphism $\phi : \tree^{(\alpha)}\to \treeq$ sends $\alpha^N(f_{i_N}) \dots \alpha^n(f_{i_n})V_{n+1}\in \gr{V}\tree^{(\alpha)}$ to the sequence $(i_m)\in \{0,\dots,q-1\}^{(-\infty,n]}\in \gr{V}\treeq$. Then the scale group acting on $\treeq$ induced by the action of $G(F,A)$ on $\tree^{(\alpha)}$ is that described in Example~\ref{examp:scale_group} with $S$ being the permutation subgroup of $\Sym(F/A)$ induced by the action $F/C\curvearrowright (F/C)/(A/C)$. 

\subsubsection{The examples}

Particular choices of $F$ and $A$ in \S\ref{sec:product_construction} illustrate the dependence of the residue group for a scale group on the choice of tidy subgroup $V$.
\begin{example}
\label{examp:Pnotunique}
Let $F=\Sym(3)$ and $A$ be the trivial subgroup. Denote the group $G(F,A)$ in \S\ref{sec:product_construction} by $G(\Sym(3))$. Then $q=|F|=6$ and the largest compact, normal $\alpha$-invariant subgroup of $G(\Sym(3))$ is trivial. The action $F\curvearrowright F/A$ is the Cayley representation of $\Sym(3)$ as a subgroup of $\Sym(6)$. To avoid confusion, this subgroup of $\Sym(6)$ will be denoted by $S_3$. 

\noindent{$\bullet$} Applying Theorem~\ref{thm:tree_rep} with $V = V_0$ represents $G(\Sym(3))\curvearrowright \tree_7$ as the scale group $P^{(S_3)}$ of Example~\ref{examp:scale_group}. The residue permutation group is $S_3$, and the level~$d$ residue group is $S_3^d$ acting on the finite tree $\tree_{6,6}^{(d)}$. 

\noindent{$\bullet$} For a different representation of $G(\Sym(3))$, let $\Alt(3)$ be the group of even permutations in $\Sym(3)$ and (denoting the identity in $\Sym(3)$ by $\iota$) apply Theorem~\ref{thm:tree_rep} with
$$
V = \left\{h\in H\mid h_m = \iota \mbox{ if }m<-1 \mbox{ and } h_{-1} \in \Alt(3)\right\}.
$$
Then $V$ is tidy for $\alpha$ because $\alpha(V)< V$. By Proposition~\ref{prop:dresidisquotient}, the residue group is $V\curvearrowright V/\alpha(V)$, which is isomorphic to $A_3\times (S_3/A_3)$ and is abelian.

To describe the residue group for this representation more explicitly, put 
$$
\Sym(3) = \langle \sigma, \tau \mid \sigma^2 = \iota = \tau^3,\ \sigma\tau\sigma = \tau^{-1}\rangle
$$
and, for $f\in\Sym(3)$, define $f_{[m]}\in H$ by $f_{[m]}(m) = f$ and $f_{[m]}(n) = \iota$ if $n\ne m$. Define maps 
$$
a,b : \Sym(3) \to H\mbox{ by }a(\sigma^i\tau^j) = \sigma^i_{[0]}\mbox{ and }b(\sigma^i\tau^j) = \tau^j_{[-1]}.
$$ 
Then $\left\{b(f)a(f)\mid f\in\Sym(3)\right\}$ is a set of coset representatives for $V/\alpha(V)$ which will enumerated $g_i$, $i\in\{0,\dots,5\}$ as in Table~\ref{tab:enum_cos_reps}.
\begin{table}[b]
	\begin{center}
		\begin{tabular}{|c||c|c|c|c|c|c|}
			\hline
			$i$&0&1&2&3&4&5\\
			\hline
			$g_i$&$\iota$ & $\tau_{[-1]}$ & $\tau^2_{[-1]}$ & $\sigma_{[0]}$ &$\tau_{[-1]}\sigma_{[0]}$ & $\tau^2_{[-1]}\sigma_{[0]}$\\
			\hline 
		\end{tabular}
	\end{center}
	\caption{Enumeration of coset representatives}
	\label{tab:enum_cos_reps}
\end{table}
Then, for $h\in H$, 
\begin{equation}
\label{eq:gproduct}
h = \prod_{m\in\mathbb{Z}} (h_m)_{[m]} = \prod_{m\in\mathbb{Z}} \alpha^m(a(h_m))\alpha^{m+1}(b(h_m)) = \prod_{m\in\mathbb{Z}} \alpha^m(g_{i_m}),
\end{equation}
which may be compared with~\eqref{eq:productforg}. 
 
The tree $\tree_7$ may now be constructed as in \S\ref{sec:tree_construction} and the action of $H$ on it computed in terms of local permutations, as in Proposition~\ref{prop:describe_autos}. For this, it suffices to compute the action of $f_{[l]}$ for $l\in\mathbb{Z}$. Doing so, we find that: for $f=\sigma^j\tau^k$ and $\vtx{v}$ in the $l$-horosphere, $\pi_{\vtx{v}}$ is the permutation 
$$
b(h_{l-1})a(h_l)\mapsto b(h_{l-1})a(fh_l)\ \mbox{ or }\ \pi_{\vtx{v}} = [
(\begin{smallmatrix} 0 & 3\end{smallmatrix})(\begin{smallmatrix} 1 & 4\end{smallmatrix})(\begin{smallmatrix} 2 & 5\end{smallmatrix})]^j;
$$
for $\vtx{v}$ in the $(l+1)$-horosphere, $\pi_{\vtx{v}}$ is the permutation 
$$
b(h_l)a(h_{l+1})\mapsto b(fh_l)a(h_{l+1})\ \mbox{ or }\ \pi_{\vtx{v}} = \begin{cases}
[(\begin{smallmatrix} 0 & 1 & 2 \end{smallmatrix})(\begin{smallmatrix} 3 & 4 & 5 \end{smallmatrix})]^k, & \mbox{ if } h_l\in  A_3 \\
[(\begin{smallmatrix} 0 & 1 & 2 \end{smallmatrix})(\begin{smallmatrix} 3 & 4 & 5 \end{smallmatrix})]^{-k}, & \mbox{ if } h_l\in \sigma A_3
\end{cases};
$$ 
and for $\vtx{v}$ in all other horospheres, $\pi_{\vtx{v}}$ if the identity map.
 
Denote the image of $G(\Sym(3))$ under its representation on $\tree_7$ by $P$. The above calculations then show that the residue group is 
$$
[P] = \langle (\begin{smallmatrix} 0 & 3\end{smallmatrix})(\begin{smallmatrix} 1 & 4\end{smallmatrix})(\begin{smallmatrix} 2 & 5\end{smallmatrix}), (\begin{smallmatrix} 0 & 1 & 2 \end{smallmatrix})(\begin{smallmatrix} 3 & 4 & 5 \end{smallmatrix})\rangle.
$$ 
In particular, $[P]$ is abelian and not isomorphic to $\Sym(3)$. Therefore $P^{(S_3)}$ and $P$ are not conjugate in $\Isom(\tree_7)_\omega$.

\noindent{$\bullet$} Other choices of tidy subgroup, one for each $r\in\mathbb{N}$, are
$$
V^{(r)} = \left\{ h\in H \mid h_m = \iota \mbox{ if }m<-r \mbox{ and } h_m\in \Alt(3)\mbox{ if }-r\leq m <0\right\},
$$
which is tidy for $\alpha$ because $\alpha(V^{(r)})<  V^{(r)}$. 

Denote the image of $G(\Sym(3))$ under its representation on $\tree_7$ by $P_r$. Then the stabiliser $(P_r)_{\vtx{v}_0}$ is isomorphic to $V^{(r)}$, and $V^{(r)}$ and $V^{(s)}$ are not conjugate if $r\ne s$. By Proposition~\ref{prop:dresidisquotient}, the level $d$ residue group of $P_r$ is isomorphic to $V^{(r)}\curvearrowright V^{(r)}/\alpha^d(V^{(r)})$ and, since $\alpha^d(V^{(r)})\triangleleft V^{(r)}$ and $S_3/A_3\cong C_2$, these groups are
$$
[P_r]_d\cong V^{(r)}/\alpha^d(V^{(r)}) \cong
\begin{cases}
A_3^{d}\times C_2^{d}, & \mbox{ if }d\leq r\\
A_3^r\times S_3^{d-r}\times C_2^r, & \mbox{ if }d>r
\end{cases}.
$$
Since they are not isomorphic for distinct values of $r$, the scale groups $P_r$ are not conjugate in $\Isom(\tree_7)_{\omega}$ although they are all isomorphic to $G(\Sym(3))$.

\noindent $\bullet$ Yet more scale group representations of $G(\Sym(3))$ may be found by considering, for $r>0$, the subgroups
$$
W^{(r)} = \left\{ h\in H \mid h_m = \iota \mbox{ if }m<-r \mbox{ and } h_m\in \langle\sigma\rangle\mbox{ if }-r\leq m <0\right\},
$$ 
which are tidy for $\alpha$ because $\alpha(W^{(r)})< W^{(r)}$. 

Denote the image of $G(\Sym(3))$ under its representation on $\tree_7$ by  $Q_r$. In this case $\alpha^d(W^{(r)})$ is not normal in $W^{(r)}$ and the level $d$ residue group $[Q_r]_d$ is isomorphic, by Proposition~\ref{prop:dresidisquotient}, to the quotient of $W^{(r)}$ by the normal core of $\alpha^d(W^{(r)})$. These groups are
$$
[Q_r]_d\cong 
\begin{cases} 
\langle\sigma\rangle^{d}\times S_3^{d}, & \mbox{ if }d\leq r\\
\langle\sigma\rangle^{r}\times S_3^{d-r}\times S_3^{r}, & \mbox{ if }d> r
\end{cases}.
$$
Hence the scale groups $Q_r$ are not conjugate for distinct~$r$ and are not conjugate to any scale group $P_r$.
\end{example}

A similar construction produces infinitely many abelian non-conjugate but isomorphic scale groups acting on $\tree_5$, as follows.
\begin{example}
\label{examp:Pnotunique2}
Let $F = C_4$, the cyclic group of order~$4$, and $A$ be the trivial subgroup. Define $G(F,A)$ as in \S\ref{sec:product_construction} and denote this group by $G(C_4)$.  Denote the order~$2$ subgroup of $C_4$ by $2C_4$ and define, for each $r\geq0$,
$$
V^{(r)} = \left\{ h\in C_4^{\mathbb{Z}} \mid h_m=0_{C_4} \mbox{ for all }m< -r,\ h_m\in 2C_4 \mbox{ if }-r\leq m <0\right\}.
$$
Denote the scale group acting on $\tree_5$ that is the image of $G(C_4)$ under the tree representation theorem by $P_r$. Then similar considerations to those above imply that
$$
[P_r]_d \cong
\begin{cases} 
(2C_4)^{d}\times (C_4/2C_2)^{d}, & \mbox{ if }d\leq r\\
(2C_4)^{r}\times C_4^{d-r}\times (C_4/2C_2)^{r}, & \mbox{ if }d> r
\end{cases}
$$
and the scale groups $P_r$ are not conjugate in $\Isom(\tree_5)_{\omega}$ for different values of~$r$ even though they are all isomorphic to $G(C_4)$.
\end{example}

\section{Groups acting doubly transitively on the boundary of a tree}
\label{sec:double_transitive}

If a scale group, $P$, is contained in a group, $G$, of isometries of $\treeq$ that does not fix the end fixed by $P$, then the action of $G$ on $\partial\treeq$ is $2$-transitive. The class of simple isometry groups of trees that are transitive on the boundary of the tree is proposed by P.-E.~Caprace in \cite{Caprace_microcosm} to be a microcosm of the class of all simple, compactly generated t.d.l.c.~groups, because many features of general groups can already be seen in these concrete isometry groups. Scale groups that are end-stabilisers of groups $2$-transitive on the boundary appear in this microcosm and it is of interest to characterise them.

Several groups, such as $\Isom(\treeq)$ itself and $PGL(2,\mathbb{Q}_p)$, are known to have $2$-transitive actions on $\partial\treeq$, and general groups with this property have been studied: in~\cite{Nebbia} in connection with harmonic analysis; and in~\cite{BM,Radu}, in connection with the structure of locally compact groups. In~\cite{Nebbia}, it is shown that if the closed subgroup $G$ is $2$-transitive on $\partial\tree$, then $G_{\vtx{w}}$ is transitive on $\partial\tree$ for every $\vtx{w}\in\gr{V}\tree$ (Proposition~1) and that $G$ has at most two orbits on $\gr{V}\tree$ (Proposition~2). See also~\cite[Lemma~3.1.1]{BM} and~\cite[Lemma~2.2]{Radu}. 
 
Results in~\cite{Nebbia} imply that the converse of the observation made in the first sentence of this section holds when $G$ is closed. 
\begin{lemma}
\label{lem:vertex_stabiliser}
Suppose that $G\leq \Isom(\tree)$ is closed and that the induced action on $\partial\tree$ is $2$-transitive. Let $\omega\in\partial\tree$. Then every $G$-orbit in $\gr{V}\tree$ is also a $G_\omega$-orbit. In particular, if the action of $G$ on $\gr{V}\tree$ is transitive, then $G_\omega$ is a scale group. 
\end{lemma}
\begin{proof}
Suppose that $\vtx{u},\vtx{v}\in \gr{V}\tree$ are in the same $G$-orbit. Then, since the action of $G$ on $\partial\tree$ has only one or two $G$-orbits in $\gr{V}\tree$ and $\tree$ is either regular or semiregular, the infinite ray $[\vtx{u},\omega)\cap[\vtx{v},\omega)$ contains an element, $\vtx{w}$, which belongs to the same $G$-orbit as $\vtx{u}$ and $\vtx{v}$.

Choose $g\in G$ with $g.\vtx{u} = \vtx{w}$. Then, since $G_{\vtx{w}}$ is transitive on $\partial\tree$ by \cite[Proposition 2]{Nebbia}, there is $h\in G$ such that $h.\vtx{w} = \vtx{w}$ and $h.(g.\omega) = \omega$. Therefore $hg$ is in $G_{\omega}$ and $(hg).\vtx{u} = \vtx{w}$. Repeating the argument for $\vtx{v}$ shows that there are $h',g'\in G_{\omega}$ such that $(h'g').\vtx{v} = \vtx{w}$. Then $(h'g')^{-1}(hg)\in G_{\omega}$ and $(h'g')^{-1}(hg).\vtx{u} = \vtx{v}$. 
\end{proof}

The following is a precise statement of the question in the first paragraph.
\begin{question}
\label{prob:2-trans?}
Which scale groups are equal to $G_{\omega}$ with $G\leq\Isom(\treeq)$ closed and $2$-transitive on $\partial\treeq$? 
\end{question}
Certain scale groups can be ruled out as end stabilisers because their residue must be a stabiliser of $q$ in a $2$-transitive permutation group on $\{0,1\dots,q\}$ and such $2$-transitive groups have been classified, see~\cite[\S7.7]{DixMort}. The remainder of this section is devoted to remarks on arguments of this kind.

A first approach is to appeal to the current structure theory of groups $G\leq\Isom(\treeq)$ transitive on $\gr{V}\treeq$ and $2$-transitive on $\partial\treeq$. The \emph{local action} of $G$ is a group $F\leq\Sym(q+1)$ comprising the permutations of $\gr{E}_{\vtx{v}}$ induced by the action of $G_{\vtx{v}}$ for a given vertex~$\vtx{v}$. This group is independent of $\vtx{v}$ up to conjugacy, see~\cite[\S0.2]{BM}. In~\cite[Lemma~3.1.1]{BM}, it is shown that $F$ is $2$-transitive and the classification of $2$-transitive finite permutation groups in~\cite{DixMort} is exploited to deduce results about the structure of $G$. 

For each group $F\leq\Sym(q+1)$ there is a vertex-transitive so-called \emph{maximal subgroup} $U(F)\leq \Isom(\treeq)$ having local action $F$. Proposition~3.2.2 in~\cite{BM} shows that, if $F$ is transitive, then every subgroup of $\Isom(\treeq)$ having local action~$F$ is contained in a conjugate of $U(F)$. If $F$ is $2$-transitive, then $U(F)$ is $2$-transitive on $\partial\treeq$. The discussion in~\cite[\S3.2]{BM} implies that $U(F)_\omega$ is isomorphic to the scale group $P^{(F_0)}$ defined in Example~\ref{examp:scale_group2}, with $F_0$ being the stabiliser of $0$ of $F\leq\Sym(q+1)$. The residue group of this scale group is~$F_0$ and, as pointed out in Remark~\ref{rem:residue}, $P^{(F_0)}$ is maximal with this residue. 

The next proposition collects parts of several results from~\cite{BM}. Recall that a group is \emph{quasisimple} if it is a perfect central extension of a simple group.
\begin{proposition}[M. Burger \& S. Mozes \cite{BM}, \S3.3]
\label{prop:local_struct}
Let the closed group $G\leq\Isom(\treeq)$ be transitive on $\gr{V}\treeq$ and $2$-transitive on $\partial\treeq$ and let $F$ be the local action of $G$. Suppose that $F_0$ is non-abelian and quasisimple. Then $G = U(F)$. 
\end{proposition}

Thus, in particular, there is a unique group up to conjugacy that is transitive on $\gr{V}\treeq$, $2$-transitive on $\partial\treeq$ and has local action~$F$. The finite $2$-transitive groups $F$ such that $F_0$ is quasisimple are listed in Examples~3.3.1 and~3.3.3 in~\cite{BM} and it follows from Proposition~\ref{prop:local_struct} and the discussion that precedes it that $P^{(F_0)}$ is the only scale group with residue $F_0$ that is a subgroup of a group $2$-transitive on $\partial\treeq$. Hence the scale groups seen in Example~\ref{examp:scale_group} having $S=F_0$ are not the stabiliser of an end in a group transitive on $\gr{V}\treeq$ and $2$-transitive on $\partial\treeq$. 

Further progress on the structure of closed groups $G\leq\Isom(\treeq)$ which are $2$-transitive on $\partial\treeq$ is made in~\cite{Radu}. In that paper, all groups whose local action contains $\Alt(q+1)$ are described. For some values of $q$ (not including prime powers) every primitive, and hence in particular every $2$-transitive, subgroup of $\Sym(q+1)$ contains $\Alt(q+1)$ and for these values of $q$ groups $2$-transitive on $\partial\treeq$ are completely classified. 

A second approach uses the observation that the group $F\leq\Sym(q+1)$ is sharply $2$-transitive if and only if~$F_0$ is a regular transitive permutation group. Finite sharply $2$-transitive groups are classified in~\cite{Zass} and shown to be affine groups of near fields~\cite{Tent}. It follows that a scale group $P$ whose residue $[P]$ is regular (sharply $1$-transitive) cannot be the end stabiliser of a group $2$-transitive on $\partial\tree$ unless $[P]$ is an action by the `multiplicative' group of a near field. The near fields are listed in~\cite{nearfield}. 

 Arguments using the classification of $2$-transitive groups do not have anything to say about the case when $q=2$ and little to say if $q<5$ because $\Alt(5)$ is the smallest non-abelian simple group. New arguments are needed to decide when a scale group is an end stabiliser for these small values of $q$ and to give a more complete solution when $q\geq5$. The closures of the Grigorchuk~\cite{Grigorchuk} and Gupta-Sidki~\cite{Gupta-Sidki} groups in $\Isom(\tree_{2,2})$ and $\Isom(\tree_{p,p})$, for $p$ an odd prime, respectively are self-replicating and more information than the residue is needed to decide whether the corresponding scale groups are end stabilisers. 
\begin{question}
\label{prob:2-trans_special}
Are the closures of the Grigorchuk and Gupta-Sidki groups equal to $G_{\omega}|_{\tree_{\vtx{v}}}$ with $G\leq\Isom(\treeq)$ closed and $2$-transitive on $\partial\treeq$? 
\end{question}

\begin{remark}
\label{rem:Rover}
Embeddings of `the' Grigorchuk group into t.d.l.c.~groups are studied in~\cite{Rover}, which describes abstract commensurators of groups acting on rooted trees. Theorem~1.2 in that paper asserts that if the group is a level transitive branch group, which applies to the Grigorchuk and Gupta-Sidki groups, then its abstract commensurator is isomorphic to its relative commensurator in the homeomorphism group of the boundary of the tree. Theorem~1.1 asserts that if the group is commensurable with its own $n^{\rm th}$ power, which also applies to the Grigorchuk and Gupta-Sidki groups, then its commensurator contains a subgroup isomorphic to a Higman-Thompson group $G_{n,1}$. Since it contains a Higman-Thompson group, the abstract commensurator cannot act by isometries on a locally finite regular tree. The results in~\cite{Rover} do not help with Question~\ref{prob:2-trans_special} because the self-replicating group is not commensurated by the scale group it embeds into under correspondence described in \S\ref{sec:s.s&s.r}. 
\end{remark}

The results in \cite{Rover} imply that the scale groups corresponding to the Grigorchuk and Gupta-Sidki groups are isomorphic to $V_{--}\rtimes\langle x\rangle$ with $x$ belonging to a compactly generated t.d.l.c.~group. This raises the following companion to Question~\ref{prob:2-trans?} outside the microcosm of groups acting on trees.
\begin{question}
	\label{prob:in_S?}
	Which scale groups are quotients of $V_{--}\rtimes\langle x\rangle$ with $G$ a compactly generated, virtually topologically simple t.d.l.c.~group and $x\in G$? 
\end{question}

Conjecture~6.1 in \cite{CapdeMedts}  proposes another criterion in answer to Question~\ref{prob:2-trans?}. The assertion conjectured implies that, if the action of $P$ on $\partial\treeq\setminus\{\omega\}$ is virtually regular, {\it i.e.\/}, if $P_{\eta}$ is finite for some end $\eta\ne\omega$, then $P$ cannot be equal to $G_{\omega}$ for any $G$ that is $2$-transitive on $\partial\treeq$. (Note that virtual regularity of $P$ on $\partial\treeq\setminus\{\omega\}$ is formally weaker than the condition discussed above that the residue $[P]$ is regular.) Conjecture~6.1 has the following companion outside the microcosm of groups acting on trees. 

\begin{Conjecture}
	\label{conj:not_virt_reg}
	Suppose that $G$ is compactly generated and topologically simple. Then, for any $x\in G$ with $q=s(x^{-1})>1$, the scale group action of $V_{--}\rtimes\langle x\rangle$ on $\tree_{q+1}$ is not virtually regular on $\partial\tree_{q+1}\setminus\{\omega\}$. 
\end{Conjecture}

It is not known whether any compactly generated, topologically simple and uniscalar t.d.l.c.~groups exist, and Conjecture~\ref{conj:not_virt_reg} says nothing about such groups. If the notion of scale group were extended to include the uniscalar groups, in which case $P\cong\mathbb{Z}$ and the tree on which $P$ is represented is a line, and Conjecture~\ref{conj:not_virt_reg} were extended to that case, then it would include the conjecture that there are no uniscalar compactly generated, topologically simple t.d.l.c.~groups.


%
%

%
%
%
\end{document}